\documentclass{amsart}
\usepackage{amsthm}
\usepackage{enumerate}

\providecommand{\MR}{\relax\ifhmode\unskip\space\fi MR }

\providecommand{\href}[2]{#2}

\theoremstyle{plain}
\newtheorem{thm}{Theorem}[section]
\newtheorem{prop}[thm]{Proposition}
\newtheorem{lemma}[thm]{Lemma}
\newtheorem{corollary}[thm]{Corollary}

\theoremstyle{definition}
\newtheorem{defin}[thm]{Definition}
\newtheorem{exa}[thm]{Example}

\newcommand{\al}{\alpha}

\newcommand{\del}{\partial}

\newcommand{\norm}[1]{\left\vert \left\vert #1\right\vert\right\vert}

\newcommand{\abs}[1]{\left\vert#1\right\vert}
\newcommand{\set}[1]{\left\{#1\right\}}

\newcommand{\N}{\ensuremath{\mathbb{N}}}

\newcommand{\R}{\ensuremath{\mathbb{R}}}
\newcommand{\Z}{\ensuremath{\mathbb{Z}}}
\newcommand{\C}{\ensuremath{\mathbb{C}}}
\title{The Inverse Laplacian: Traces in Infinite Dimensions}
\author{Bryce Morrow}


\begin{document}

\begin{abstract}
In this paper, we present a concise development of the well-studied theory of trace class operators on infinite dimensional (separable) Hilbert spaces suitable for an advanced undergraduate, as well as a construction of the inverse Laplacian on closed manifolds. With these developments acting as prerequisite, we present original trace computations involving the inverse Laplacian on the (flat) torus, generalizing computations done by R. Grady and O. Gwilliam within the context of topological quantum field theory.

\end{abstract}

\maketitle

\tableofcontents

\section{Finite-Dimensional Traces}

Let us start our discussion on infinite-dimensional traces with a familiar subject: traces in finite dimensions. Let $V$ be a vector space of dimension $n \in \N$ over $\C$. Also, let $P:V \to V$ be a linear map. Choose a basis $\set{e_i}_{i=1}^{n} \subset V$ for $V$ and let $[P]_{\set{e_i}} \in \C^{n \times n}$ denote the matrix determined by the operator $P$ under the basis $\set{e_i}_{i=1}^n$. We define the \textit{trace} of $P$ as 
$$\text{Tr}(P) := \text{Tr}\left([P]_{\set{e_i}}\right)$$
where $\text{Tr}\left([P]_{\set{e_i}}\right)$ denotes the trace of the matrix $[P]_{\set{e_i}}$ given as the sum of the diagonal entries.\\

\begin{prop}
$\text{Tr}(P)$ is well-defined.    
\end{prop} 

\begin{proof}
In order for $\text{Tr}(P)$ to be well-defined, it must be shown that the trace of $P$ is independent of the chosen basis for $V$. That is, let $\set{f_i}_{i=1}^n \subset V$ be another basis for $V$. It must be shown that $\text{Tr}([P]_{\set{e_i}}) = \text{Tr}([P]_{\set{f_i}})$. Let $C$ be the change-of-basis matrix from $\set{e_i}_{i=1}^n$ to $\set{f_i}_{i=1}^n$. Note that $[P]_{\set{e_i}} = C^{-1} [P]_{\set{f_i}} C$. Also, for any $A,B \in \C^{n \times n}$, note that $\text{Tr}(AB) = \text{Tr}(BA)$. Hence,

$$\text{Tr}([P]_{\set{e_i}}) = \text{Tr}(C^{-1} [P]_{\set{f_i}} C) = \text{Tr}([P]_{\set{f_i}} C C^{-1}) = \text{Tr}([P]_{\set{f_i}}). $$

\end{proof}
Now, assume that $V$ is equipped with an inner product $\langle \cdot, \cdot \rangle$ with the linearity condition imposed on the second argument (note that by identifying $V$ with $\C^n$, we can always assume that $V$ is endowed with an inner product). Let $\set{e_j}_{j=1}^n$ be an orthonormal basis for $V$, and let $p_{ij}$ denote the $ij^{th}$ entry of $[P]_{\set{e_j}}$. That is, $P(e_j) = \sum_{i=1}^n p_{ij} e_i$. Then
\begin{equation*}
\begin{split}
\left\langle e_j, P(e_j) \right\rangle & = \left\langle e_j, \: \sum_{i=1}^n p_{ij} e_i \right\rangle\\
& = \sum_{i=1}^n p_{ij} \langle e_j, e_i \rangle\\
& = p_{jj}.
\end{split}
\end{equation*}
Since $$\text{Tr}\left([P]_{\set{e_j}}\right) = \sum_{j=1}^n p_{jj},$$
we have
\begin{equation}
\begin{split}
\text{Tr}(P)  = \sum_{j=1}^n \left\langle e_j, P(e_j) \right\rangle.\label{trace equation} 
\end{split}
\end{equation}
This is precisely the characterization of the trace that we will use to generalize the definition to infinite dimensions.

Conveniently, we may define the trace of $P$ without explicitly making reference to a basis. Let $\set{\lambda_i}_{i=1}^n \subset \C$ be the eigenvalues, with multiplicity, of $P$, so for any $z \in \C$, the eigenvalues of $I + z[P]_{\set{e_i}}$ are $\set{1 + z\lambda_i}_{i=1}^n$. Recall that the determinant of a matrix can be computed as the product of its eigenvalues, so we have
$$\det \left(I + z[P]_{\set{e_i}}\right) =  \prod_{i=1}^{n}\left(I + z \lambda_i\right).$$
The identities
$$\dfrac{d}{d z} 
 \det \left(I + z[P]_{\set{e_i}}\right) \bigg \vert_{z = 0} = \: \text{Tr}\left( [P]_{\set{e_i}}\right)$$
 and
$$\dfrac{d}{d z}\prod_{i=1}^{n}\left(I + z \lambda_i\right) \bigg \vert_{z=0} = \: \sum_{i=1}^n \lambda_i$$
imply that $\text{Tr}\left([P]_{\set{e_i}}\right) = \sum_{i=1}^N \lambda_i$.
Hence, 
$$\text{Tr}(P) = \sum_{i = 1}^n \lambda_i.$$ 

For certain classes of operators on countably-infinite-dimensional (Hilbert) spaces, this characterization of the trace as the sum of eigenvalues agrees with the formulation given in (\ref{trace equation}). Specifically, these formulations agree for trace class operators via Lidskii's theorem, the proof of which is essentially the same as the argument given above for finite dimensional operators.

\section{Infinite-Dimensional Traces}\label{infinite-dimensional traces}
Before proceeding with our discussion of traces, we must clarify the setting in which we will be working. Thus, some recalling/defining must be done.
\subsection{Hilbert Spaces}

Let $V$ be a vector space over $\C$. Recall that an \textit{inner product} on $V$ is a map $\langle \cdot , \cdot \rangle: V \times V \to \C$ such that for all $x,y,z \in V$ and $a,b \in \C$

\begin{itemize}
    \item $\langle x, y \rangle = \overline{\langle y, x \rangle}$ 
    \item $\langle x, ay + bz \rangle = a\langle x, y \rangle + b \langle x, z \rangle$ \footnote{Sometimes, this linearity condition is imposed on the first argument rather than the second.}
    \item $\langle x,x \rangle \in \R$ with $\langle x,x \rangle \geq 0$
    \item $\langle x,x \rangle = 0 $ if and only if $x = 0$
\end{itemize}

We call the space $V$ endowed with an inner product $\langle \cdot, \cdot \rangle$ an \textit{inner product space}. Given an inner product, we define the norm on $V$ to be the map $\norm{-}: V \to [0, \infty)$ given by 
$$\norm{x} = \sqrt{\langle x, x \rangle}$$ for all $x \in V$.

If an inner product space $V$ over $\C$ is complete with respect to the metric induced by $\norm{-}$, then we say that $V$ is a \textit{Hilbert space}. Although one can consider real Hilbert spaces, it is common to restrict attention to complex Hilbert spaces, which is what we shall do in this section.

Recall that, for some indexing set $B$, a subset $\set{\varphi_n}_{n \in B} \subset V$ is an \textit{orthonormal set} if $\langle \varphi_n, \varphi_m \rangle = 0$ when $n \neq m$ and $\langle \varphi_n, \varphi_m \rangle = 1$ if $n = m$. If $V$ is a Hilbert space and $B$ is countably-infinite or finite, then we say that an orthonormal set $\set{\varphi_n}_{n \in B}$ is an \textit{orthonormal basis} for $V$ if 
\begin{equation}
x = \sum_{n \in B} \langle \varphi_n, x\rangle \varphi_n \label{basis def}
\end{equation}
for all $x \in V$ where the sum converges to $x$ with respect to the metric induced by the norm on $V$ (no matter how the terms are ordered). Requiring that the indexing set $B$ is countable is not necessary; however, we will be restricting our attention to countable orthonormal bases in this paper. Recall that if a Hilbert space admits one countably-infinite orthonormal basis, then all other orthonormal bases must also be countably-infinite. Likewise, if a Hilbert space admits one finite basis, all other bases are finite. Hence, the notion of a finite-dimensional or countably-infinite-dimensional Hilbert space is well-defined.

In order to show that an orthonormal set is an orthonormal basis, it is often convenient to use the following equivalencies.

\begin{thm}  Let $\set{\varphi_n}_{n \in B} \subset V$ be a countable orthonormal set for the Hilbert space $V$. The following are equivalent.\label{equivalencies}
\end{thm} \begin{enumerate}[\itshape 1.]
    \item \textit{$\set{\varphi_n}_{n \in B}$ is an orthonormal basis (according to (\ref{basis def})).}
    \item \textit{If $\langle \varphi_n, x \rangle = 0$ for all $n \in B$, then $x = 0$.}
    \item \textit{(Parseval's Identity) For all $x \in V$, $\norm{x}^2 = \sum_{n \in B} \abs{ \langle \varphi_n, x \rangle }^2.$}
\end{enumerate}

Another particularly useful result for Hilbert spaces (more generally inner product spaces) that we will be utilizing throughout the rest of this paper is the Cauchy--Schwarz inequality.

\begin{thm}[Cauchy--Schwarz Inequality]
Let $V$ be an inner product space. For all $x,y \in V$
$$\abs{\langle x,y \rangle} \leq \lVert x \rVert \cdot \lVert y \rVert.$$\label{Cauchy--Schwarz}
\end{thm}

To conclude our discussion on Hilbert spaces, recall that, for two Hilbert spaces $\mathcal{H}$ and $\mathcal{W}$, a linear map $A:\mathcal{H} \to \mathcal{W}$ is \textit{bounded} if there exists $C \in [0,\infty)$ such that $\norm{Ax}_{\mathcal{W}} \leq C \norm{x}_{\mathcal{H}}$ for all $x \in \mathcal{H}.$ We denote the space of bounded linear maps from $\mathcal{H}$ to $\mathcal{W}$ as $L(\mathcal{H},\mathcal{W})$. Also, we let $L(\mathcal{H})$ denote the space of bounded linear maps from $\mathcal{H}$ to $\mathcal{H}$. Recall that a linear map $A: \mathcal{H} \to \mathcal{W}$ is bounded iff it is continuous iff it is continuous at $0 \in \mathcal{H}$. Here, continuity is respect to the topologies induced by the inner products on $\mathcal{H}$ and $\mathcal{W}$.

Further review on Hilbert spaces, including proofs of theorems \ref{equivalencies} and \ref{Cauchy--Schwarz}, can be found in many analysis textbooks, including \cite{Folland}.

\subsection{Trace Class Operators}\label{trace class opertors}

For the remainder of section \ref{infinite-dimensional traces}, unless stated otherwise, let $\mathcal{H}$ denote a complex Hilbert space that has a countable orthonormal basis (so every orthonormal basis for $\mathcal{H}$ is countable). We will let $N \in \N \cup \set{\infty}$ denote the dimension of $\mathcal{H}$. 

Given a bounded operator $A \in L(\mathcal{H})$, we would like to define the trace of $A$ as
\begin{equation}
\text{Tr}(A) = \sum_{n=1}^{N} \langle \varphi_n, A \varphi_n \rangle \label{trace equation but again}  
\end{equation}
\noindent where $\set{\varphi_n}_{n=1}^N \subset \mathcal{H}$ is some orthonormal basis for $\mathcal{H}$. Recall that on finite-dimensional Hilbert spaces, (\ref{trace equation but again}) agrees with the usual definition. Unlike finite-dimensional operators, the trace of an operator on an infinite-dimensional Hilbert space may not be a well-defined complex number according to the definition posed above. Consider the identity map $I$ on $\mathcal{H}$ when the dimension of $\mathcal{H}$ is countably-infinite. For any orthonormal basis $\set{\varphi_j}_{j=1}^\infty \subset \mathcal{H}$,
$$\text{Tr}(I) = \sum_{j=1}^\infty \langle \varphi_j, I \varphi_j \rangle = \sum_{j=1}^\infty \langle \varphi_j, \varphi_j \rangle = \sum_{j=1}^\infty 1,$$
which diverges to $\infty$.

The following example is even more problematic. Consider the space of all (complex) square summable sequences
$$l^2 := \set{ \set{a_n}_{n=1}^\infty \subset \C \; : \; \sum_{n=1}^\infty |a_n|^2 < \infty }$$
and the map $T:l^2 \to l^2$ given by 
$$T(\set{a_n}_{n=1}^\infty) = \set{(-1)^n a_n}_{n=1}^\infty.$$

It is straightforward to show that $T$ is linear and bounded. Now let $\set{\varphi_j}_{j=1}^{\infty} \subset l^2$ be the standard orthonormal basis for $l^2$. That is, $\varphi_j$ is the sequence where the $j^{th}$ term is $1$ and the rest of the terms are $0$. E.g. $\varphi_1 = (1,0,0,0,...), \varphi_2 = (0,1,0,0,...),$ and $ \varphi_3 = (0,0,1,0,...)$. Then 
$$\text{Tr}(T) = \sum_{n=1}^\infty \langle \varphi_j, T \varphi_j \rangle = \sum_{n=1}^\infty \langle \varphi_j, (-1)^j \varphi_j \rangle = \sum_{j=1}^\infty (-1)^j,$$
which is clearly not a well-defined quantity.

Given the issues encountered with the two previous examples, it is evident that we must be careful in determining which operators we can consider the trace of. The remainder of this section is dedicated towards defining the aptly named trace class operators, which, as we shall see, are the \textit{only} operators with well-defined traces. As a prerequisite, we start by defining the operators $A^\ast$, $\sqrt{A}$, and $\abs{A}$.

\begin{thm}[Riesz representation theorem] Let $\mathcal{H}$ be a Hilbert space. For all bounded linear functionals $f: \mathcal{H} \to \C$ there exists a unique $y_f \in \mathcal{H}$ such that $f(x) = \langle y_f, x \rangle$ for all $x \in \mathcal{H}.$

\end{thm}

\begin{defin}
Let $\mathcal{H}$ and $\mathcal{W}$ be Hilbert spaces and $A:\mathcal{H} \to \mathcal{W}$ a bounded linear map. Define the \textit{adjoint} of $A$ to be the unique bounded linear function $A^\ast : \mathcal{W} \to \mathcal{H}$ satisfying 

$$\langle w, Ax \rangle_{\mathcal{W}} = \langle A^\ast w , x \rangle_{\mathcal{H}}$$
for all $x \in \mathcal{H}$ and $w \in \mathcal{W}$.
\end{defin}

\begin{prop}
The adjoint is well-defined.
\end{prop}

\begin{proof}
Let $\mathcal{H}, \mathcal{W}$, and $A$ be as in the previous definition. To show that the map $A^\ast$ exists, we will utilize the Riesz representation theorem.

Let $w \in \mathcal{W}$ and define the map $f_w: \mathcal{H} \to \C$ by $f_w(x) = \langle w, Ax \rangle_{\mathcal{W}}$ for all $x \in \mathcal{H}$. Since the maps $x \mapsto \langle w, x \rangle_{\mathcal{W}}$ and $A$ are linear and bounded, $f_w$ must be bounded and linear. In particular, this implies that $f_w$ is a bounded functional. Hence, by the Riesz representation theorem, there exists a unique $y_{f_{w}} \in \mathcal{H}$ satisfying
$$f_w(x) = \langle y_{f_{w}}, x \rangle_{\mathcal{H}} \iff \langle w, Ax \rangle_{\mathcal{W}} = \langle y_{f_{w}}, x \rangle_{\mathcal{H}} $$
for all $x \in \mathcal{H}$. Define the map $A^\ast: \mathcal{W} \to \mathcal{H}$ by $A^\ast(w) = y_{f_{w}}$ where $y_{f_{w}}$ is defined as above for all $w \in \mathcal{W}$. It immediately holds that $A^\ast$ is the unique map satisfying \\$\langle w, Ax \rangle_{\mathcal{W}} = \langle A^\ast (w), x \rangle_{\mathcal{H}}$. All that is left to show is that $A^\ast$ is linear and bounded. Linearity of $A^\ast$ is fairly straightforward to show using the linearity of inner products in the second argument. $A^\ast$ can be shown to be bounded using the Cauchy–Schwarz inequality. 
\end{proof}

\begin{exa}
    Consider the case when $\mathcal{H} = \C^n$ and $\mathcal{W} = \C^m$. Let $A: \C^n \to \C^m$ be a linear map. Chose a basis for each complex space, and consider $A$ to be a matrix. Recall that for $v,u \in \C^n$, the inner product of $v$ and $u$ can be written as $v^\dagger u$ where $\dagger$ denotes the conjugate transpose. For any $x \in \C^n$ and $w \in \C^m$, 
    $$\langle w, Ax \rangle = w^\dagger Ax = (A^\dagger w)^\dagger x = \langle A^\dagger w, x \rangle.$$
    Since $A^\ast$ is the unique (bounded) linear map that satisfies the equation above, it must be that $A^\ast = A^\dagger$. That is, as a matrix, the adjoint of $A$ is the conjugate transpose of itself. If the entries of $A$ are all real, then $A^\ast$ is the transpose of $A$.
\end{exa}

\begin{defin}
    Let $A$ be a bounded operator on $\mathcal{H}$. $A$ is \textit{positive-semidefinite} (denoted $A \geq 0$) if for all $x \in \mathcal{H}$, $\langle x , Ax \rangle \in \R$ with $\langle x , Ax \rangle \geq 0$.
\end{defin}

\begin{thm}[Square Root Lemma]
    Let $A$ be a bounded operator on $\mathcal{H}$ that is positive semi-definite. Then there exists a unique $B \in L(\mathcal{H})$  such that $B$ is positive semi-definite and $B^2 = A$. We define $\sqrt{A} := B$.\label{sqrt lemma}
\end{thm}
The key idea to prove \ref{sqrt lemma} is to assume that $\norm{A} \leq 1$ (by replacing $A$ with a scalar multiple of itself) so that $I + \sum_{j=1}^\infty c_j (I - A)^j$ converges absolutely with respect to the operator norm on $L(\mathcal{H})$. Here, the $c_j$ are such that $1 + \sum_{j=1}^\infty c_j z$ is the power series expansion of $\sqrt{1-z}$ about the origin in $\C$, and $I$ is the identity map on $\mathcal{H}$. Define $B:= I + \sum_{j=1}^\infty c_j (I - A)^j$, and rearrange the terms to show that $B^2 = A$. A full proof can be found in chapter 6 of \cite{math physics}.

Note that by definition of the adjoint operator, $$\langle (A^\ast A) x, x \rangle = \langle A^\ast (A x), x \rangle = \langle Ax, Ax \rangle = \norm{Ax}^2 \in [0,\infty).$$

In particular, by taking the complex conjugate it follows that \\$\langle x, (A^\ast A)x \rangle \in [0,\infty).$ Hence, $A^\ast A$ is positive semi-definite (in fact, $A^\ast A$ is positive definite). Thus, we can define the following.

\begin{defin}
    Let $A \in L(\mathcal{H})$. Define the \textit{absolute value} of $A$ to be the positive semi-definite operator

    $$|A|: = \sqrt{A^\ast A}.$$
\end{defin}

Finally, we are in a good position to define the trace class condition.

\begin{defin}[Trace Class]
    Let $A \in L(\mathcal{H})$. $A$ is \textit{trace class} if there exists an orthonormal basis $\set{\varphi_n}_{n =1}^N$ such that 

    $$\sum_{n=1}^N \langle \varphi_n, |A| \varphi_n \rangle < \infty.$$

    Note that since $|A|$ is positive semi-definite, each term in the series above is nonnegative. Thus, the series either converges to a finite real number, or it diverges to infinity.
\end{defin}

In order to prove that an operator $A$ is \textit{not} trace class, one need not consider all possible orthonormal bases of a Hilbert space. It is sufficient to check if $\sum_{n=1}^N \langle \varphi_n, |A| \varphi_n \rangle = \infty$ for a particular orthonormal basis $\set{\varphi_n}$. This is due to the following proposition.

\begin{prop}
    Let $A \in L(\mathcal{H})$. Then for any two orthonormal bases $\set{\varphi_n}_{n=1}^N$ and $\set{\psi_n}_{n=1}^N$ 

    $$\sum_{n=1}^N \langle \varphi_n, |A| \varphi_n \rangle = \sum_{n=1}^N \langle \psi_n, |A| \psi_n \rangle.$$
\end{prop}

\begin{proof}

Utilizing both Parseval's identity and the fact that $\langle x, |A| x \rangle = \norm{\sqrt{|A|}x}^2$, we get

\begin{equation*}
\begin{split}
    \sum_{n=1}^N \langle \varphi_n, |A| \varphi_n \rangle & = \sum_{n=1}^N \norm{\sqrt{|A|} \varphi_n}^2\\ 
    & = \sum_{n=1}^N\sum_{j=1}^N \abs{\left\langle \psi_j, \sqrt{|A|} \varphi_n \right\rangle}^2\\ 
    & = \sum_{n=1}^N\sum_{j=1}^N \abs{\left\langle \varphi_n, \left(\sqrt{|A|}\right)^\ast \psi_j \right\rangle}^2\\ 
    & = \sum_{j=1}^N\sum_{n=1}^N \abs{\left\langle \varphi_n, \left(\sqrt{|A|}\right)^\ast \psi_j \right\rangle}^2\\ 
    & = \sum_{j=1}^N \norm{\left(\sqrt{|A|}\right)^\ast \psi_j}^2\\ 
\end{split}
\end{equation*}
where the order of the sums above can be interchanged since the summands are non-negative. By considering the case for which $\set{\varphi_n}_{n=1}^N$ is replaced by $\set{\psi_j}_{j=1}^N$, we get 
$$\sum_{n=1}^N \langle \psi_n, |A| \psi_n \rangle = \sum_{j=1}^N \norm{\left(\sqrt{|A|}\right)^\ast \psi_j}^2$$
as well, so
    $$\sum_{n=1}^N \langle \varphi_n, |A| \varphi_n \rangle = \sum_{n=1}^N \langle \psi_n, |A| \psi_n \rangle.$$
\end{proof}

Finally, in the following proposition, we have the result that motivates the use of the definition for trace class we presented.
\begin{prop}
    Let $A \in L(\mathcal{H})$ be trace class, and let $\set{\varphi_n}_{n=1}^{N}$ be an orthonormal basis for $\mathcal{H}.$ The series
    $$\sum_{n=1}^N \langle \varphi_n, A \varphi_n \rangle$$
    converges absolutely. Furthermore, if $\set{\psi_n}_{n=1}^N$ is another orthonormal basis, then

    $$\sum_{n=1}^N \langle \varphi_n, A \varphi_n \rangle = \sum_{n=1}^N \langle \psi_n, A \psi_n \rangle.$$
    I.e.  the trace of $A$ exists. \label{trace class well-defined trace}
\end{prop}

\begin{defin}[Trace]
    Let $A \in L(\mathcal{H})$ be trace class and $\set{\varphi_n}_{n=1}^N \subset \mathcal{H}$ an orthonormal basis. We define the \textit{trace} of $A$ as
    $$\text{Tr}(A):= \sum_{n=1}^N \langle \varphi_n, A \varphi_n  \rangle.$$
\end{defin}

\begin{proof} \textit{(of Proposition \ref{trace class well-defined trace})} \;
A trace class operator $A$ can be expressed as
\begin{equation}
A = \sum_{n=1}^{N(A)} \mu_n(A) \langle \al_n, \cdot \rangle \beta_n, \label{canonical decomp} 
\end{equation}
where $\set{\al_n}$ and $\set{\beta_n}$ are orthonormal sets, $\mu_n(A) > 0$ are eigenvalues of $|A|$, and \\$N(A) \in \N \cup \set{\infty}$. (\ref{canonical decomp}) is called the canonical decomposition of $A$.\footnote{This decomposition holds for all compact operators, not just those which are trace class.} See chapter 6 of \cite{math physics} for further details. 
By linearity and continuity of $\langle \cdot, \cdot \rangle$ in the second argument, we get
\begin{equation*}
\begin{split}
\sum_{m=1}^N \langle \varphi_m, A\varphi_m \rangle & = \sum_{m=1}^N \sum_{n=1}^{N(A)} \mu_n(A) \langle \varphi_m, \beta_n \rangle \langle \al_n, \varphi_m \rangle.
\end{split}    
\end{equation*}

Since $$\sum_{m=1}^N \abs{\langle \varphi_m, \beta_n \rangle  \langle \al_n, \varphi_m \rangle} =  \sum_{m=1}^N \abs{\langle \varphi_m, \beta_n \rangle}  \abs{\langle \al_n, \varphi_m \rangle} \leq \lVert \beta_n\rVert \cdot \lVert \al_n \rVert = 1$$ 

and 

$$\sum_{n=1}^{N(A)}\mu_n(A) < \infty,$$
we have
\begin{equation*}
\begin{split}
\sum_{m=1}^N \sum_{n=1}^{N(A)} \abs{\mu_n(A) \langle \varphi_m, \beta_n \rangle \langle \al_n, \varphi_m \rangle} & =  \sum_{n=1}^{N(A)} \mu_n(A) \sum_{m=1}^{N} \abs{\langle \varphi_m, \beta_n \rangle \langle \al_n, \varphi_m \rangle}\\
& \leq \sum_{n=1}^{N(A)} \mu_n(A)\\
& < \infty.
\end{split}
\end{equation*}
Thus, the iterated series converges absolutely (and so does the original series) so we can switch the order of the iterated sum. Utilizing the fact that $\sum_{m=1}^{N} \langle \varphi_m, x \rangle \langle y, \varphi_m \rangle = \langle x,y \rangle$ for any $x,y \in \mathcal{H}$, we have

\begin{equation*}
\begin{split}
\sum_{m=1}^N \langle \varphi_m, A\varphi_m \rangle & = \sum_{n=1}^{N(A)} \sum_{m=1}^N  \mu_n(A) \langle \varphi_m, \beta_n \rangle \langle \al_n, \varphi_m \rangle\\
& = \sum_{n=1}^{N(A)} \mu_n(A) \langle \beta_n, \al_n \rangle.
\end{split}
\end{equation*}

Since the value of the series $\sum_{m=1}^{N} \langle \varphi_m, A \varphi_m \rangle$ does not depend on the chosen orthonormal basis $\set{\varphi_n}$, the proof is complete.

\end{proof}

For an operator $A \in L(\mathcal{H})$ to have a well-defined trace, we require that, for any orthornormal bases $\set{\varphi_n}$ and $\set{\psi_n}$ of $\mathcal{H}$, 
$\sum_{n=1}^N \langle \varphi_n, A \varphi_n \rangle = \sum_{n=1}^N \langle \psi_n, A \psi_n \rangle$ (where this quantity is finite). In particular, by reordering $\set{\varphi_n}$, we require that $\sum_{n=1}^N \langle \varphi_n, A \varphi_n \rangle$ converges to the same value no matter how the terms are arranged. This is equivalent to saying that $\sum_{n=1}^N \langle \varphi_n, A \varphi_n \rangle$ converges absolutely. That is, for $A$ to have a well-defined trace, the series $\sum_{n=1}^N \langle \varphi_n, A \varphi_n \rangle$ must necessarily converge absolutely for any orthonormal basis $\set{\varphi_n}$. It turns out this condition is equivalent to $A$ being trace class, implying that trace class operators are the only operators with well-defined traces.

\begin{lemma}
    Let $A \in L(\mathcal{H})$ be self-adjoint. That is, $A^\ast = A$. Then, $A$ is trace class if and only if for every orthonormal basis $\set{\varphi_n}_{n=1}^N \subset \mathcal{H}$ the sum $\sum_{n=1}^N \langle \varphi_n, A \varphi_n \rangle$ converges absolutely.\label{self adjoint trace equivalency}
\end{lemma}

\begin{proof}
 The ``only if'' implication has already been proved in proposition \ref{trace class well-defined trace}. Suppose that $\sum_{n=1}^N \abs{ \langle \varphi_n, A \varphi_n \rangle} < \infty$ for every orthonormal basis $\set{\varphi_n}_{n=1}^N$. Our goal is to define two positive semi-definite operators $A_+$ and $A_-$ such that $A = A_+ - A_-$, along with an orthonormal basis $\set{\varphi_n} \cup \set{\psi_n}$ for $\mathcal{H}$ such that $A_+\psi_n = 0$ and $A_- \varphi_n = 0$. Letting $\set{e_m}_{m \in \N \sqcup \N} = \set{\varphi_n}_{n \in \N} \cup \set{\psi_n}_{n \in \N}$ denote this orthonormal basis, we will then have
 \begin{equation*}
 \begin{split}
\sum_{m \in \N \sqcup \N} \langle e_m, A_+ e_m \rangle & = \sum_{n=1}^N \abs{\langle  \varphi_n , A \varphi_n \rangle}\\
& \leq \sum_{m \in \N \sqcup \N} \abs{\langle e_m, A e_m \rangle}\\
& < \infty
 \end{split}
 \end{equation*}
and similarly, 
\begin{equation*}
\begin{split}
\sum_{m \in \N \sqcup \N} \langle e_m, A_- e_m \rangle & = \sum_{n=1}^N \abs{ \langle \psi_n, A \psi_n  \rangle}\\
& \leq \sum_{m \in \N \sqcup \N} \abs{\langle e_m, A e_m \rangle}\\
& < \infty
\end{split}
\end{equation*}
since $A\varphi_n = A_+ \varphi_n$ and $A \psi_n = A_- \psi_n$. As $A_+$ and $A_-$ are both positive semi-definite, $|A_+| = A_+$ and $|A_-| = A_-$, so this will imply that $A_+$ and $A_-$ are both trace class. The set of trace class operators on $\mathcal{H}$ forms a vector space (see chapter 6, section 6 of \cite{math physics}) so it will follow that $A = A_+ - A_-$ is trace class.
 
Define $A_+:= \frac{1}{2}\left( |A| + A \right)$ and $A_-:= \frac{1}{2}\left( |A| - A \right)$, so $A = A_+ - A_-$. Now let $P_-$ be the orthogonal projection onto $\text{Ker}(A_+)$ and $P_+:= I - P_-$, where $\text{Ker}(A_+)$ is the kernel of $A_+$ and $I$ is the identity on $\mathcal{H}$. Note that $P_+ = I - P_-$ is an orthogonal projection since $P_-$ is an orthogonal projection. Then $(|A| + A)P_- = 0$, so by taking the adjoint of each side, we have
\begin{equation*}
\begin{split}
\big((|A| + A)P_-\big)^\ast = 0 & \implies P_-^\ast(|A|^\ast + A^\ast) = 0\\
& \implies P_-(|A| + A) = 0,
\end{split}
\end{equation*}
since positive semi-definite operators (i.e. $P_-$ and $|A|$) are self-adjoint, and $A$ is self-adjoint by assumption. Also, since $A |A| = |A|A$ and $|A|^2 = \left(\sqrt{A^2}\right)^2 = A^2$, we have $A_+ A_- =  0$, implying that $\text{Im}(A_-) \subset \text{Ker}(A_+)$ where $\text{Im}(A_-)$ is the image of $A_-$. Hence, for any $x \in \mathcal{H}$, $A_- (x) \in \text{Ker}(A_+)$, so $P_- (A_-(x)) = A_-(x)$. That is, $P_-A_- = A_-$, which clearly implies that
$$P_- (|A| - A) = |A| - A.$$
The above equations imply that $2 P_- A = P_-(|A| + A) - P_-(|A| - A) = A - |A|$, so
$$|A| = (I - 2P_-)A.$$
By taking the adjoint of both sides, we also have
$$|A| = A(I - 2P_-).$$
Likewise, $2P_- |A| = P_-(|A| + A) + P_-(|A| - A) = |A| - A$, so 
$$A = (I - 2P_-)|A|,$$
and
$$A = |A|(I - 2P_-).$$
By replacing $A$ with $(I - 2P_-)|A|$ in $A_+ = \frac{1}{2}(|A| + A)$, we get
\begin{equation*}
\begin{split}
A_+ & = \frac{1}{2}(|A| + A)\\
& = \frac{1}{2}(|A| + (I - 2P_-)|A|)\\
& = (I - P_-)|A|\\
& = P_+ |A|.
\end{split}
\end{equation*}
By replacing $A$ with $|A|(I - 2P_-)$ instead, we have
$$A_+ = |A|P_+.$$
By a similar argument, 
$$A_- = P_- |A| = |A|P_-.$$
Since $A_+$ and $A_-$ are the product of two commuting, positive semi-definite operators, they are both positive semi-definite.

Let's now construct the orthonormal basis $\set{\varphi_n} \cup \set{\psi_n}$. By definition of $P_+$ and $P_-$,
 $$P_+ + P_- = I.$$
Given this equality, we claim that $\text{Ker}\left(P_-\right) = \text{Im}\left(P_+\right)$  (since $P_+$ is an orthogonal projection, $\text{Im}\left(P_+\right)$ is simply the subspace $P_+$ projects onto). Indeed, let \\$x \in \text{Ker}\left(P_-\right)$. Then $P_-x + P_+x = x$ implies that $P_+x = x \implies x \in \text{Im}\left( P_+ \right).$ Now let $y \in \text{Im}\left( P_+ \right)$, so $y = P_+x$ for some $x \in \mathcal{H}$. From the decomposition $\mathcal{H} = \text{Im}\left( P_-\right) \oplus \left(\text{Im}\left( P_-\right)\right)^\perp$, we have $x = u + v$ for some $u = P_-x \in \text{Im}\left( P_-\right)$ and $v \in \left(\text{Im}\left( P_-\right)\right)^\perp$. Since $P_-$ is an orthogonal projection, it's clear that $ \left(\text{Im}\left( P_-\right)\right)^\perp = \text{Ker}\left( P_-\right)$. Hence, $v \in \text{Ker}\left( P_-\right)$, so 
\begin{equation*}
\begin{split}
P_+x + P_-x = x & \implies P_+x + u = u + v\\
& \implies P_+x = v\\
& \implies y = v\\
& \implies y \in \text{Ker}\left( P_-\right).
\end{split}
\end{equation*} 

Since we have shown that $\text{Ker}\left(P_-\right) = \text{Im}\left(P_+\right)$, and it also holds that \\$\text{Ker}\left( P_+\right) = \left(\text{Im}\left( P_+\right)\right)^\perp$ (since, once again, $P_+$ is an orthogonal projection) the decomposition $\mathcal{H} = \text{Im}\left( P_+ \right) \oplus \left(\text{Im}\left( P_+ \right) \right)^\perp$ is equivalent to $\mathcal{H} = \text{Ker}\left(P_-\right) \oplus \text{Ker}\left( P_+\right)$. Let $\set{\varphi_n}_{n=1}^N$ be an orthonormal basis for $\text{Ker}\left(P_-\right)$ and $\set{\psi_n}_{n=1}^N$ an orthonormal basis for $\text{Ker}\left( P_+\right)$, so $\set{\varphi_n} \cup \set{\psi_n}$ is an orthonormal basis for $\mathcal{H}$ with $P_-\varphi_n = 0$ and $P_+\psi_n = 0.$ Recall $A_+ = |A|P_+$ and $A_- = |A|P_-$, so $A_- \varphi_n = 0$ and $A_+ \psi_n = 0$.

Since we have constructed $A_+,A_- \geq 0$ with $A = A_+ - A_-$ such that $A_+ \psi_n = 0$ and $A_- \varphi_n = 0$, the proof is complete by the argument above.
\end{proof}

\begin{prop}
    Let $A \in L(\mathcal{H})$. $A$ is trace class if and only if for every orthonormal basis $\set{\varphi_n}_{n=1}^N \subset \mathcal{H}$ the sum 
    $\sum_{n=1}^N \langle \varphi_n, A \varphi_n \rangle$
    converges absolutely.\label{trace class equivalency}
\end{prop}

\begin{proof}
Let $\text{Re}(A): = \frac{1}{2}(A + A^\ast)$ and $\text{Im}(A):= \frac{1}{2i}(A - A^\ast)$. Note that 
$$A = \text{Re}(A) + i \text{Im}(A).$$
Suppose that $\sum_{n=1}^N \abs{ \langle \varphi_n, A \varphi_n \rangle} < \infty$ for every orthonormal basis $\set{\varphi_n}_{n=1}^N$. Then
\begin{equation*}
\begin{split}
 \sum_{n=1}^N \abs{ \langle \varphi_n, A^\ast \varphi_n \rangle} & =  \sum_{n=1}^N \abs{ \langle A \varphi_n, \varphi_n \rangle}\\ 
 & = \sum_{n=1}^N \abs{ \langle \varphi_n, A \varphi_n \rangle}\\
 & < \infty.
\end{split}
\end{equation*}
Thus,
\begin{equation*}
\begin{split}
\sum_{n=1}^N \abs{\langle \varphi_n, (A + A^\ast)\varphi_n   \rangle } & = \sum_{n=1}^N \abs{\langle \varphi_n, A \varphi_n \rangle + \langle \varphi_n, A^\ast \varphi_n \rangle}\\
& \leq \sum_{n=1}^N \abs{\langle \varphi_n, A \varphi_n \rangle} + 
 \sum_{n=1}^N \abs{\langle \varphi_n, A^\ast \varphi_n \rangle}\\
& < \infty.
\end{split}
\end{equation*}
Similarly, $\sum_{n=1}^N \abs{\langle \varphi_n, (A - A^\ast) \varphi_n  \rangle} < \infty$, and any scalar multiple of $A + A^\ast$ and $A - A^\ast$ also satisfies this property. In particular, \begin{equation*}
\sum_{n=1}^N \abs{\langle \varphi_n, \text{Re}(A) \varphi_n \rangle} < \infty  
\end{equation*}
and
$$\sum_{n=1}^N \abs{\langle \varphi_n, \text{Im}(A) \varphi_n \rangle} < \infty.$$
Since $\text{Re}(A)$ and $\text{Im}(A)$ are self-adjoint, they are both trace class by lemma \ref{self adjoint trace equivalency}. Therefore, $A = \text{Re}(A) + i \text{Im}(A)$ is trace class as well.
\end{proof}

Note that, in general, the sum $\sum_{n=1}^N \langle \varphi_n, A \varphi_n \rangle$ converging absolutely for \textit{one} orthonormal basis does not imply that $A$ is trace class. This sum must converge absolutely for \textit{all} orthonormal bases to guarantee $A$ is trace class. Let's look at an example illustrating this point.
\begin{exa}
 Consider once again $l^2$, the Hilbert space of square summable sequences of complex numbers. Let $\set{\varphi_n}_{n=1}^\infty$ denote the standard basis on $l^2$. E.g. $\varphi_1 = (1,0,0,0,...), \varphi_2 = (0,1,0,0,...)$, and $\varphi_3 = (0,0,1,0,0,...)$. Now define 
$$\psi_n := \frac{1}{\sqrt{n(n+1)}}\left( \sum_{i=1}^n \varphi_i - n\varphi_{n+1} \right)$$
for each $n \in \N$. E.g. $\psi_1 = \frac{1}{\sqrt{2}}(1, -1, 0, 0,...)$, $\psi_2 = \frac{1}{\sqrt{6}}(1,1,-2,0,0,...)$, and \\$\psi_3 = \frac{1}{\sqrt{12}}(1,1,1,-3,0,0,...)$. Showing that $\set{\psi_n}_{n=1}^\infty$ is an orthonormal set is a straightforward calculation. Let's show that it is an orthonormal basis. Let $a_{\bullet} \in l^2$ be such that $\langle \psi_n, a_\bullet \rangle = 0$ for all $n \in \N$. $\langle \psi_1, a_\bullet \rangle = \frac{1}{\sqrt{2}} \cdot a_1 - \frac{1}{\sqrt{2}}a_2$, so $\langle \psi_1, a_\bullet \rangle = 0$ implies that $a_2 = a_1$. Similarly, 
$$\langle \psi_2, a_\bullet \rangle = \frac{1}{\sqrt{6}}(a_1 + a_2 - 2a_3) = \frac{1}{\sqrt{6}}(a_1 + a_1 - 2a_3) = \frac{1}{\sqrt{6}}(2a_1 - 2a_3),$$
so $\langle \psi_2, a_\bullet \rangle = 0$ implies that $a_3 = a_1$. Via an induction argument, it follows that $a_n = a_1$ for all $n$. The only way $a_\bullet = (a_1, a_1, a_1,...)$ can be square summable is if $a_1 = 0$. That is, $a_\bullet = 0$. Therefore, $\set{\psi_n}$ is an orthonormal basis for $l^2$ by theorem \ref{equivalencies}.

Now let $L:l^2 \to l^2$ be the left-shift operator defined as
$$L\left(\set{a_n}_{n=1}^\infty\right) = \set{a_{n+1}}_{n=1}^\infty.$$
For example, $L\varphi_1 = (0,0,0,...)$, $L\varphi_2 = (1,0,0,0,...)$, and $L\varphi_3 =  (0,1,0,0,0,...)$. Note that 
$$\sum_{n=1}^\infty \langle \varphi_n, L \varphi_n \rangle = 0,$$
but
$$\sum_{n=1}^\infty \langle \psi_n, L \psi_n  \rangle = \sum_{n=1}^\infty -\frac{1}{n(n+1)}.$$
Even though both of these series converge absolutely, $$\sum_{n=1}^\infty \langle \varphi_n, L \varphi_n \rangle \neq \sum_{n=1}^\infty \langle \psi_n, L \psi_n  \rangle,$$ so $L$ is not trace class.\\
\end{exa}

As previously alluded to, trace class operators are subject to Lidskii's theorem--- a very natural result for calculating traces considering its presence in finite dimensions.
\begin{thm}[Lidskii's Theorem] Let $A \in L(\mathcal{H})$ be trace class with eigenvalues $\set{\lambda_n}_{n=1}^{N(A)}$ such that each eigenvalue is repeated according to its algebraic multiplicity. Then 

$$\text{Tr}(A) = \sum_{n=1}^{N(A)} \lambda_n.$$    
\end{thm} 
A proof of Lidskii's theorem can be found in \cite{Operator Theory}.

\subsection{Diagonalizable Operators}

Suppose $A \in L(\mathcal{H})$ is trace class. Suppose further that there exists an orthonormal basis $\set{v_n}_{n=1}^{N} \subset \mathcal{H}$ such that each $v_n$ is an eigenvector of $A$ with eigenvalue $\lambda_n$. That is, $\set{v_n}_{n=1}^{N}$ is an eigenbasis with respect to the operator $A$. Then, without needing to invoke Lidskii's theorem, the trace of $A$ is calculated as
\begin{equation}
\text{Tr}(A) = \sum_{n=1}^N \langle v_n, A v_n \rangle = \sum_{n=1}^N \langle v_n, \lambda_n v_n \rangle = \sum_{n=1}^N \lambda_n \langle v_n, v_n \rangle = \sum_{n=1}^N \lambda_n. \label{equation for calculating trace using eigenbasis} 
\end{equation}

In the finite dimensional case, matrices corresponding to an operator for which there exists an eigenbasis are said to be diagonalizable. We extend this nomenclature to operators on infinite dimensional Hilbert spaces as well.

Diagonalizable operators adhere to a converse of (\ref{equation for calculating trace using eigenbasis}) in that, if we have an eigenbasis for an operator, we can use it to determine if said operator is trace class, as is shown in proposition \ref{compute trace with eigvals}. This result provides an effective method to calculate traces that will be utilized extensively in section \ref{traces involving inverse laplace}.

\begin{prop}
    Let $A \in L(\mathcal{H})$ and suppose $\set{e_m}_{m=1}^N \subset \mathcal{H}$ is an (orthonormal) eigenbasis for $A$ with corresponding eigenvalues $\set{\lambda_m}_{m=1}^N$ such that $\sum_{m=1}^N \lambda_m$ converges absolutely. Then $A$ is trace class. Moreover,
    $$\text{Tr}(A) = \sum_{m=1}^N \lambda_m.$$\label{compute trace with eigvals}
\end{prop}

\begin{proof}
    Let $\set{\varphi_n}_{n=1}^N \subset \mathcal{H}$ be an orthonormal basis. For each $n,m \in \N$, let \\$a_m^n := \langle e_m, \varphi_n \rangle$ so that $\varphi_n = \sum_{m=1}^N \langle e_m, \varphi_n \rangle e_m = \sum_{m=1}^N a_m^n e_m$ by expanding $\varphi_n$ in terms of the orthonormal basis $\set{e_m}$. Then
    \begin{equation*}
    \begin{split}
    \sum_{n=1}^N \langle \varphi_n, A \varphi_n \rangle & = \sum_n \left\langle \varphi_n,\: A \left(\sum_{m=1}^N a_m^n e_m\right) \right\rangle  \\
    & = \sum_n \left\langle \varphi_n,\:  \sum_{m} a_m^n A e_m \right\rangle\\
    & = \sum_n \left\langle \varphi_n,\:  \sum_{m} a_m^n \lambda_m e_m \right\rangle.
    \end{split}
    \end{equation*}
Where the second line follows from the fact that $A$ is a linear and bounded. Since $\langle \cdot, \; \cdot \rangle$ is linear and continuous in the second argument, we get

\begin{equation*}
\begin{split}
\sum_{n} \langle \varphi_n, A \varphi_n \rangle & =  \sum_n  \sum_{m} \lambda_m  \left\langle \varphi_n,\: a_m^n e_m \right\rangle\\
& = \sum_n  \sum_{m} \lambda_m  \overline{\left\langle a_m^n e_m, \: \varphi_n \right\rangle}.
\end{split}
\end{equation*}
By expanding $\varphi_n$ once more and utilizing the fact that complex conjugation is continuous and distributes over sums, we achieve
\begin{equation*}
\begin{split}
\sum_{n} \langle \varphi_n, A \varphi_n \rangle & = \sum_n \sum_m \lambda_m \overline{  \left\langle a_m^n e_m, \sum_{j=1}^N a_j^n e_j    \right\rangle  }\\
& = \sum_n \sum_m \lambda_m \sum_j \overline{\langle a_m^n e_m, a_j^n e_j  \rangle}\\
& = \sum_n \sum_m \lambda_m \sum_j \langle a_j^n e_j, a_m^n e_m   \rangle.
\end{split}
\end{equation*}
Note that $\langle a_j^n e_j, a_m^n e_m \rangle = a_m^n \cdot \overline{a_j^n} \langle e_j, e_m \rangle$, so $\langle a_j^n e_j, a_m^n e_m  \rangle = 0$ if $j \neq m$, and \\$\langle a_m^n e_m, a_m^n e_m \rangle = a_m^n \cdot \overline{a_m^n} = |a_m^n|^2$ since $\set{e_m}$ is an orthonormal basis. Thus,
\begin{equation*}
\begin{split}
\sum_n \langle \varphi_n, A\varphi_n  \rangle & = \sum_n \sum_m \lambda_m |a_m^n|^2.
\end{split}
\end{equation*}
By plugging in $\langle e_m, \varphi_n \rangle$ for $a_m^n$, we get
$$\sum_n \langle \varphi_n, A\varphi_n  \rangle = \sum_n \sum_m \lambda_m |\langle e_m, \varphi_n  \rangle  |^2.$$
By Parseval's Identity (applied with respect to the basis $\set{\varphi_n}$) 
\begin{equation*}
\begin{split}
\sum_n \sum_m \abs{\lambda_m \cdot \abs{\langle e_m, \varphi_n  \rangle}^2 } & = \sum_m \sum_n \abs{\lambda_m} \cdot \abs{\langle e_m, \varphi_n  \rangle}^2\\
& = \sum_m \abs{\lambda_m} \sum_n \abs{\langle e_m, \varphi_n  \rangle}^2\\
& = \sum_m \abs{\lambda_m} \cdot \norm{e_m}^2\\
& = \sum_m |\lambda_m| \\
& < \infty \; \text{ (By assumption)}.
\end{split}
\end{equation*}
Thus, $\sum_n \langle \varphi_n, A \varphi_m  \rangle$ converges absolutely, so we can swap the order of the sums above and apply Parseval's Identity once more to get
\begin{equation*}
\begin{split}
\sum_n \langle \varphi_n, A \varphi_n  \rangle & = \sum_n \sum_m \lambda_m \abs{\langle e_m, \varphi_n  \rangle  }^2\\
& = \sum_m \lambda_m \sum_n \abs{\langle e_m, \varphi_n  \rangle  }^2\\
& = \sum_m \lambda_m \norm{e_m}^2\\
& = \sum_m \lambda_m.
\end{split}
\end{equation*}
Since the basis $\set{\varphi_n}$ was arbitrary, the proof is complete by proposition \ref{trace class equivalency}.

\end{proof}

\section{The Inverse Laplace Operator} 
The purpose of this section is to define the inverse Laplace operator, which we will be studying in the final section within the context of traces. The following notations will be fixed throughout the rest of this section.
\begin{itemize}
    \item $(M,g)$ will denote a smooth, oriented, closed, connected, $n$-dimensional Riemannian manifold.
    \item $C^\infty(M)$ is the space of smooth maps from $M$ to $\R$.
    \item $\Gamma_{C^\infty}(TM)$ will denote the space of smooth vector fields on $M$.
    \item $L^2_\R(M)$ will denote the real Hilbert space of real-valued square integrable  functions on $M$.\footnote{Integration on $M$ is with respect to the Borel measure induced by the Riemannian volume form on $M$.}
    \item $L^2_\C(M)$ will denote the complex Hilbert space of complex-valued square integrable functions on $M$.
\end{itemize}

For reference on general manifold language, see \cite{Manifolds 1}. 
\subsection{The Laplacian}

Recall the following definitions, all of which can be found in chapter 2 of \cite{Old and New Spectral Geom}.

\begin{defin}
The \textit{gradient} is the linear map \label{grad def}
$$\text{grad}: C^{\infty}(M) \to \Gamma_{C^{\infty}}(TM)$$
such that, for any $\varphi \in C^\infty(M)$ and any $p \in M$, $(\text{grad} \varphi)_p$ is the unique tangent vector at $p$ such that 
$$\left\langle X_p, (\text{grad}\varphi)_p \right\rangle_{g(p)} = X_p(\varphi)$$
for all tangent vectors $X_p$ at $p.$
\end{defin}

In definition \ref{grad def} above, $\langle \cdot, \cdot \rangle_{g(p)}$ is the inner product on the tangent space at $p \in M$ induced by the metric $g$.

\begin{defin}
The \textit{divergence} operator is the  linear map
$$\text{div}: \Gamma_{C^\infty}(TM) \to C^\infty(M)$$
given by
$$\text{div}(X) = \text{trace}\left( \nabla X \right)$$
for all $X \in \Gamma_{C^\infty}(TM)$ where $\nabla X$ is the covariant derivative of $X$.
\end{defin} 

\begin{defin}
The \textit{Laplace operator} (or the Laplacian) on $M$ is the linear operator
$$\Delta: C^\infty(M) \to C^\infty(M)$$
given by 
\begin{equation*}
\Delta = \text{div} \circ \text{grad}.\footnote{Often, the Laplacian is defined as the negation of how we have defined it here. That is, it is typical to define $\Delta_M:= - \text{div} \circ \text{grad}$. Doing so makes $\Delta_M$ a positive operator.}
\end{equation*}   \label{Laplace def}
\end{defin}
A subscript will be added to $\Delta$ whenever we would like to emphasize the manifold in which it is acting on. For example, $\Delta_M:C^\infty(M) \to C^\infty(M)$.

The previous three definitions given all agree with the definitions one would encounter in a multivariable calculus course when $M = \R^n$. When $M \neq \R^n$, the Laplacian takes a more familiar form in terms of partial derivatives when utilizing local coordinates. For any chart $(U, x_1,...,x_n)$ of $M$, we have

\begin{equation}
\Delta f \big \vert_U = \dfrac{1}{\sqrt{\det G}} \sum_{i,j = 1}^n \dfrac{\del}{\del x_i} \left( g^{ij} \sqrt{\det G} \: \dfrac{\del f}{\del x_j} \right) \label{local coords laplace}  
\end{equation}
for each $f \in C^\infty(M)$. Here, $g^{ij}$ is the entry in the $i^{th}$ row and $j^{th}$ column of $G^{-1}$, where $G$ denotes the metric tensor matrix in the given local coordinates.  
\begin{exa}
Let's take $M = S^{1}$. Throughout this paper, define $S^1 := \R/2\pi\Z$. Let $q:\R \to S^1 = \R/2\pi\Z$ be the quotient map $q(x) = [x]$ where $[x]$ denotes the equivalence class of $x$ in $S^1$. The restrictions $q \big \vert_{(0,2\pi)}: (0,2\pi) \to (0,2\pi)$ (where the codomain is viewed as a subspace $(0,2\pi) \subset S^1$) and $q\big \vert_{(-\pi, \pi)}: (-\pi, \pi) \to S^1 \setminus \set{[\pi]}$ are invertible. These inverses are the coordinate functions for charts that cover $S^1$, so by definition of a smooth map, $f:S^1 \to \R$ is smooth if and only if $f \circ q \big \vert_{(0,2\pi)}$ and $f \circ q \big \vert_{(-\pi,\pi)}$ are smooth. Since $f \circ q$ is $2\pi$-periodic, $f \circ q \big \vert_{(0,2\pi)}$ and $f \circ q \big \vert_{(-\pi,\pi)}$ are smooth if and only if $f \circ q$ is smooth, so we have 
\begin{equation}
C^\infty(S^1) = \set{f:S^1 \to \R \: \bigg \vert \: f \circ q \in C^\infty(\R)}.\label{categorizing all smooth maps on circle}    
\end{equation}

Let $\frac{d}{dx}:C^\infty(\R) \to C^\infty(\R)$ denote the derivative map between smooth functions on $\R$. We claim that for every $f \in C^\infty(S^1)$ there exists a unique map $h \in C^\infty(S^1)$ such that $\frac{d(f \circ q)}{dx} = h \circ q$. Since $f \circ q$ is $2\pi$-periodic, $(f \circ q)(x + 2 \pi m) = (f \circ q)(x)$ for all $x \in \R$ and $m \in \Z$. Taking the derivative of both sides yields
$$\frac{d(f \circ q)}{dx}(x + 2\pi m) = \frac{d(f\circ q)}{dx}(x)$$
by the chain rule. That is, $\frac{d(f \circ q)}{dx}$ is $2\pi$-periodic as well. Hence, the map $h:S^1 \to \R$ given by 
$h\left([x]\right) = \frac{d(f \circ q)}{dx}(x)$ is well-defined, and we have $h \circ q = \frac{d(f \circ q)}{dx}$ as desired. Smoothness of $h$ follows from (\ref{categorizing all smooth maps on circle}) and the fact that $\frac{d(f \circ q)}{dx}$ is smooth. The uniqueness of $h$ is immediate from that fact that $k = \widetilde{k}$ if and only if $k \circ q = \widetilde{k} \circ q$ for any two maps $k,\widetilde{k}:S^1 \to \R$.

Let $\frac{d}{d\theta}:C^\infty(S^1) \to C^\infty(S^1)$ denote the map assigning $f \mapsto h$ as above. That is,
$$\frac{d(f \circ q)}{dx} = \frac{df}{d \theta} \circ q.$$
$\frac{d}{d \theta}$ is the derivative on $\R$ descended to $S^1$. With respect to a chart given as the inverse of $q$ when appropriately restricted, the canonical metric $g$ on $S^1$ is such that $G = 1$. Hence, it follows from (\ref{local coords laplace}) that
$$\Delta_{S^1} = \frac{d^2}{d\theta^2}.$$
That is, the Laplace on $S^1$ according to definition \ref{Laplace def} is exactly as we would expect it to be defined; it is the Laplace on $\R$ descended to $S^1$ in the sense that
$$\Delta_{\R}\left(f \circ q \right) = \Delta_{S^1}(f) \circ q.$$ \label{S^1 laplace}
\end{exa}

\begin{exa}
Let's now consider $M = S^1 \times S^1$. Let $q \times q: \R^2 \to S^1 \times S^1$ denote the map $(x_1,x_2) \mapsto (q(x_1),q(x_2))$ where, once again, $q:\R \to S^1$ is the quotient map. By an argument similar to that in the last example,
$$C^\infty(S^1 \times S^1) = \set{f: S^1 \times S^1 \to  \R \: \bigg \vert \: f \circ (q \times q) \in C^\infty\left(\R^2\right)}.$$
Let $\frac{\del}{\del x_1}$ and $\frac{\del}{\del x_2}$ denote the partial derivatives on $\R^2$ with respect to the first and second arguments respectively. Similar to the previous example, there are unique maps $\frac{\del}{\del \theta_1}, \frac{\del}{\del \theta_2}: C^\infty(S^1 \times S^1) \to C^\infty(S^1 \times S^1)$ which descend the partial derivatives on $\R^2$ to $S^1 \times S^1$ in the sense that 
$$\frac{\del\left(f \circ (q\times q)\right)}{\del x_1} = \frac{\del f}{\del \theta_1} \circ (q \times q)$$
and
$$\frac{\del\left(f \circ (q\times q)\right)}{\del x_2} = \frac{\del f}{\del \theta_2} \circ (q \times q)$$
for all $f \in C^\infty(S^1 \times S^1)$.
With respect to a chart given as the inverse of $q \times q$ when appropriately restricted, $S^1 \times S^1$ is ``flat'' in the sense that $g_{ij} = \delta_{ij}$ where $\delta_{ij}$ is the Kronecker delta. Hence, 
$$\Delta_{S^1 \times S^1} = \frac{\del^2}{\del \theta_1^2} + \frac{\del^2}{\del \theta_2^2}.$$
One can easily check that
$$\Delta_{\R^2}\left(f \circ (q \times q) \right) = \Delta_{S^1 \times S^1}(f) \circ (q \times q).$$\\ \label{torus laplace}
\end{exa}

Since we have assumed that $M$ is closed, the Laplacian is formally self-adjoint in the following sense.

\begin{prop}
(Green's Identity) Let $f,\varphi \in C^{\infty}(M)$. Then
$$\int_M f \cdot \Delta \varphi = \int_M \langle \text{grad}f, \text{grad}\varphi  \rangle_g = \int_M \varphi \cdot \Delta f.$$\label{green's identity}
\end{prop}

A proof of proposition \ref{green's identity} can be found in \cite{Old and New Spectral Geom}. The term formally self-adjoint comes from realizing that the map $(f,\varphi) \mapsto \int_{M} f\cdot \varphi$ is the inner product for $L_\R^2(M)$, so if $\Delta$ was an operator on $L_\R^2(M)$, this formula would imply that $\Delta$ is its own adjoint.

\subsection{Defining the Inverse Laplace Operator}

For $u \in C^\infty(M)$ and $j \in \N_0$, let $\nabla^j u$ denote the $j^{th}$ covariant derivative of $u$ where, by convention, $\nabla^0 u = u$ and $\nabla^1 u = \nabla u$. One can think of the covariant derivatives of a function as being coordinate independent presentations of partial derivatives. That is, $\nabla u$ can be thought of as containing information regarding the first partial derivatives, $\nabla^2 u$ contains the second partial derivatives, and so on.

\begin{defin}[Sobolev Space]
    $H^k(M)$ is the real Hilbert space given as the Cauchy completion of $C^{\infty}(M)$ equipped with the inner product
    $$\langle u,v \rangle_{H^k}: = \sum_{j=0}^k \int_M \langle \nabla^j u, \nabla^j v \rangle_g$$
    for $u,v \in C^\infty(M)$ where $\langle \cdot, \cdot \rangle_g$ is the pointwise inner product on covariant tensor fields induced by the metric $g$. $H^k(M)$ is the $k^{th}$ \textit{Sobolev space} of $M$.
\end{defin}

By convention, $\langle u, v \rangle_g = u \cdot v$. For $j = 1$, note that  $\langle \nabla u, \nabla v \rangle_g = \langle \text{grad}u, \text{grad}v \rangle_g$, so for $u,v \in C^\infty(M)$
\begin{equation*}
\begin{split}
\langle u, v \rangle_{H^1} & = \int_M u \cdot v  + \int_M \langle \nabla u, \nabla v \rangle_g\\
& = \int_M u \cdot v  + \int_M \langle \text{grad}(u), \text{grad}(v) \rangle_g. 
\end{split}
\end{equation*}
Hence, proposition \ref{green's identity} implies that  
\begin{equation}
\begin{split}
\langle u, v \rangle_{H^1} & = \int_M u \cdot v + \int_M \Delta u \cdot v\\
& = \int_M u \cdot v + \int_M u \cdot \Delta v\label{ideal sobolev inner product smooth}  
\end{split}
\end{equation}
for smooth functions $u$ and $v$.

Note that a sequence of smooth functions $\set{f_n}_{n=1}^\infty \subset C^\infty(M)$ that is Cauchy with respect the norm on $H^k(M)$ is also Cauchy in $L^2_\R(M)$. In particular, $\set{f_n}$ converges to some $f \in L^2_\R(M)$ with respect to the norm on $L^2_\R(M)$. Also, a Cauchy sequence that converges to $0$ in $L_\R^2(M)$ also converges to $0$ in $H^k(M)$, so the map $H^k(M) \hookrightarrow L^2_\R(M)$ given by sending $\big[\set{f_n}_n\big] \in H^k(M)$ to its limit $f$ in $L^2_\R(M)$ is a well-defined, linear injection. Hence, we identify $H^k(M)$ with its image in $L^2_\R(M)$ under this map. A more detailed description on this identification can be found in \cite{Old and New Spectral Geom}.

Now, define the bilinear map $B:C^\infty(M) \times C^\infty(M) \to \R$ by 
$$B(u,v) = \int_M \Delta u \cdot v.$$
As explained in \cite{Old and New Spectral Geom} $B$ extends to a symmetric bilinear map $B:H^1(M) \times H^1(M) \to \R$ by defining
$$B(u,v) = \underset{k \to \infty}{\lim} B(u_k, v_k)$$
for sequences of smooth functions $\set{u_k}_{k=1}^{\infty}$ and $\set{v_k}_{k=1}^\infty$ converging to $u$ and $v$ respectively in $H^1(M)$. It follows from this definition that the extension of $B$ is continuous.

Note that for any $u \in H^k(M)$ we have 
\begin{equation}
\norm{u}_{L^2} \leq \norm{u}_{H^k} \label{compare L^2 and H^k}
\end{equation}
since this inequality holds for smooth functions. As a corollary, the inclusion \\$H^k(M) \hookrightarrow L^2(M)$ is continuous. As a second corollary, the map $H^1(M) \times H^1(M) \to \R$ sending 
$$(u,v) \mapsto \int_M u \cdot v$$
is continuous since, by the Cauchy--Schwarz inequality, we have
$$\left| \int_M u \cdot v \right | = \left | \langle u,v  \rangle_{L^2}   \right | \leq \norm{u}_{L^2} \cdot \norm{v}_{L^2} \leq \norm{u}_{H^1} \cdot \norm{v}_{H^1}. \footnote{Here, continuity is respect to the product topology on $H^1(M)\times H^1(M)$. Continuity of a bilinear form $B$ is equivalent to it being bounded: there exists some $c > 0$ such that \\$\abs{B(u,v)} \leq c \norm{u} \cdot \norm{v}$ for all $u$ and $v$.}
$$
Thus, it follows that for any $u,v \in H^1(M)$,
\begin{equation}
\begin{split}
\langle u,v  \rangle_{H^1} & = \underset{k \to \infty}{\lim} \langle u_k,v_k  \rangle_{H^1}\\\label{sobolev inner product final}
& = \underset{k \to \infty}{\lim} \left( \int_M u_k \cdot v_k  + \int_M \Delta u_k \cdot v_k \right)\\
& = \underset{k \to \infty}{\lim} \int_M u_k \cdot v_k  + \underset{k \to \infty}{\lim} B(u_k,v_k )\\
& = \int_M u \cdot v + B(u,v),
\end{split}
\end{equation}
where once again $\set{u_k}_k$ and $\set{v_k}_k$ are sequences of smooth functions converging to $u$ and $v$ with respect to the $H^1$ norm. Equation (\ref{sobolev inner product final}) generalizes (\ref{ideal sobolev inner product smooth}) in the sense that they agree on $C^\infty(M) \times C^\infty(M)$.

Utilizing its eigenfunctions, the Laplace operator admits an extension to a linear map $\Delta: H^2(M) \to L^2_\R(M)$ called the Friedrich extension: There exist eigenfunctions $\set{\varphi_j}_{j=1}^\infty \subset C^\infty(M)$ of $\Delta$ (with eigenvalues $\set{\lambda_j}_{j=1}^\infty$) that form an orthonormal basis for $L^2_\R(M)$. For $f = \sum_{j=1}^\infty \langle \varphi_j, f\rangle_{L^2} \cdot \varphi_j \in H^2(M)$ the Friedrich extended Laplacian $\Delta f$ is defined as
$$\Delta f = \sum_{j=1}^\infty \lambda_j \langle \varphi_j, f\rangle_{L^2} \cdot \varphi_j.$$ See \cite{Old and New Spectral Geom} for more details on the extension. Of particular interest to us is the following proposition. 

\begin{prop}
    For any $u \in H^2(M)$ and $v \in H^1(M)$\label{connecting B and Laplace on H^2}
    $$B(u,v) = \int_M \Delta u \cdot v $$
\end{prop}
A proof of proposition \ref{connecting B and Laplace on H^2} can be found in \cite{Old and New Spectral Geom}. 

Our goal is now to utilize the bilinear form $B$ in order to define the inverse Laplace operator. Let $K \subset H^1(M)$ be the linear subspace
\begin{equation}
K: = \set{u \in H^1(M) \; : \; B(u,v) = 0 \; \text{ for all } v \in H^1(M)}.\label{definition of K}    
\end{equation} 
We will define the inverse Laplacian, denoted by $D^{-1}$, such as to satisfy the equality 
$$B\left(D^{-1}(h), v\right) = \int_M h \cdot v$$
for all $v \in K^\perp$, where $K^\perp$ is the orthogonal compliment of $K$ in $H^1(M)$. 

In particular, if $D^{-1}(h) \in H^2(M)$, then the equality above implies that
$$\int_M \Delta\left(D^{-1}(h)\right) \cdot v = \int_M h \cdot v,$$
showing that $D^{-1}$ is ``distributionally'' a right-sided inverse to $\Delta$ on the preimage of $H^2(M)$ under $D^{-1}$ with test functions in $K^\perp$. We first present a couple necessary results.

\begin{prop}
    Let $H$ be a Hilbert space and $C \subset H$ a closed, linear subspace. $H/C \cong C^\perp$ isometrically. \label{isometric prop}
\end{prop}

\begin{proof}
First, recall that $H/C$ is a normed vector space with 
$$\norm{x + C} = \underset{u \in C}{\inf} \norm{x - u}$$
for all $x + C \in H/C$. Define the linear map $T: C^\perp \to H/C$ by 
$$T(x) = x + C$$ 
for all $x \in C^\perp.$ If $T(x) = 0$, then $x+C = 0$, implying that $x \in C$, so $x \in C \cap C^\perp \implies x = 0$, which shows that $T$ is injective. Now suppose $y + C \in H/C$. To show surjectivity, recall that $H = C \oplus C^\perp$, so $y = u + v$ for some $u \in C$ and $v \in C^\perp$. In particular, $y + C = v + C$. Hence, $T(v) = v + C = y + C$ showing that $T$ is surjective.

To prove that $T$ is an isometry, we must show that $\norm{x} = \underset{u \in C}{\inf }\norm{x - u}$ for all $x \in C^\perp.$ Once again, the decomposition $H = C \oplus C^\perp$ implies that $x = u_0 + v_0$ for some $u_0 \in C$ and $v_0 \in C^\perp$. In particular,  $u_0$ and $v_0$ are the unique vectors in $C$ and $C^\perp$ respectively satisfying $\norm{x - u_0} = \underset{u \in C}{\inf }\norm{x - u}$ and $\norm{x - v_0} = \underset{v \in C^\perp}{\inf }\norm{x - v}$. Since $x \in C^\perp$, we can take $v = x$ to get 
$$\norm{x - v_0} = \underset{v \in C^\perp}{\inf }\norm{x - v} = \norm{x-x} = 0,$$ so $v_0 = x$. It follows from the equation $x = u_0 + v_0$ that $u_0 = 0$. Therefore, 
$$\norm{x} = \norm{x - u_0} = \underset{u \in C}{\inf }\norm{x - u}$$
\end{proof}

\begin{thm}[Lax--Milgram] Let $\mathcal{H}$ be a Hilbert space and $B:\mathcal{H} \times \mathcal{H} \to \R$ a continuous bilinear map that is coercive. By coercive, we mean there exists some $c > 0$ such that for all $u \in \mathcal{H}$ it holds that $B(u,u) \geq c \norm{u}^2$. Then for any bounded linear functional $f: \mathcal{H} \to \R$, there exists a  unique $u_f \in \mathcal{H}$ so that 
$$B(u_f,v) = f(v)$$
for all $v \in \mathcal{H}$. Furthermore, it holds that 
$$\norm{u_f} \leq \dfrac{1}{c}\norm{f}.$$
\end{thm}

See chapter 6 of \cite{Evans} for a proof of the Lax--Milgram theorem. 
Let \\$\widetilde{B}: K^\perp \times K^\perp \to \R$ be the restriction of the bilinear form $B$ to $K^\perp \times K^\perp$ where $K^\perp$ is the orthogonal compliment of $K$ in $H^1(M)$. It is a standard result that the orthogonal compliment of any subset of an inner product space is a closed subspace of said inner product space: Suppose $\set{u_k}_{k=1}^\infty \subset K^\perp$ converges to $u \in H^1(M)$. Then for any $v \in K$, the continuity of $\langle \cdot, \cdot \rangle_{H^1}$ in the second argument gives
$$\langle  v, u \rangle_{H^1} = \langle  v,  \underset{k \to \infty}{\lim}u_k \rangle_{H^1} = \underset{k \to \infty}{\lim} \langle v, u_k \rangle_{H^1} = 0,$$
implying that $u \in K^\perp$, showing $K^\perp \subset H^1(M)$ is closed. 

Hence, $K^\perp$ is a Hilbert space, so the Lax--Milgram theorem can be applied to $\widetilde{B}$ if we can show that it is continuous and coercive. 

\begin{prop}
The bilinear form $\widetilde{B}:K^\perp \times K^\perp \to \R$ is continuous and coercive.
\end{prop}
\begin{proof}
Recall that $B$ is continuous, so it follows immediately that $\widetilde{B}$ is continuous. Let's now show that $\widetilde{B}$ is coercive. Let $u \in K^\perp$. $K \subset H^1(M)$ is closed, so it follows from proposition \ref{isometric prop} that $K^\perp \cong 
H^1(M)/K $ isometrically. In particular, we have 
$$\norm{u}_{H^1} = \underset{v \in K}{\inf} \norm{u - v}_{H^1} \leq \norm{u - \overline{u}}_{H^1},$$
where $\overline{u} = \frac{1}{Vol(M)}\int_M u$ is the average value of $u$ over $M$ (a constant function). Let \\$\set{u_k}_k \subset C^\infty(M)$ be a sequence of smooth functions converging to $u$ in $H^1(M)$. Then $\set{u_k - \overline{u}}_k$ converges to $u - \overline{u}$, and 
\begin{equation*}
\begin{split}
  \norm{u}_{H^1}^2 & \leq \norm{u - \overline{u}}_{H^1}^2\\
  & = \int_M(u - \overline{u})^2 + B(u - \overline{u}, u - \overline{u})\\
  & = \int_M(u - \overline{u})^2 + \underset{k \to \infty}{\lim} \left(\int_M \Delta(u_k - \overline{u})\cdot (u_k - \overline{u})\right)\\
  & = \int_M (u - \overline{u})^2 + \underset{k \to \infty}{\lim} \left(\int_M \Delta u_k \cdot u_k - \overline{u}\int_M \Delta u_k \right)\\
  & = \int_M (u - \overline{u})^2 + B(u,u).
\end{split}
\end{equation*}

For any $v \in C^\infty(M)$, we let $\abs{\nabla v} = \sqrt{\langle \nabla v, \nabla v \rangle_g}$ denote the pointwise norm of $\nabla v$. It's clear that $\int_M \abs{\nabla v}^2 = B(v,v)$. When $v \in H^1(M)$, we define $\abs{\nabla v}$ to be the limit of $\abs{\nabla v_k}$ in $L^2_\R(M)$ where $\set{v_k}_{k=1}^\infty \subset C^\infty(M)$ converges to $v$ in $H^1(M)$. The assignment $H^1(M) \to L_\R^2(M)$ sending $v \mapsto \abs{\nabla v}$ is continuous, so it follows that 
\begin{equation}
B(v,v) = \int_M \abs{\nabla v }^2\label{B = integral of abs of nabla u^2}
\end{equation}
holds for all $v \in H^1(M)$ since it is true for smooth functions. The Poincaré inequality for $H^1(M)$ states that there exists some $A > 0$ so that 
$$\int_M (v - \overline{v})^2 \leq A\int_M |\nabla v|^2$$
for any $v \in H^1(M)$. See section 2.8 in \cite{Sobolev Spaces} for this inequality.\footnote{Note that as presented in \cite{Sobolev Spaces}, this inequality only holds when the dimension of $M$ is greater than one since we are utilizing $L^2$ Sobolev spaces. The requirement that $\text{dim}(M) > 1$ is only needed to utilize the Rellich--Kondrakov theorem, stating that the inclusion \\$H^1(M) \hookrightarrow L^2_\R(M)$ is compact. However, when $\text{dim}(M) =1 $, the inclusion is still compact. By (ii) of theorem 2.9 in \cite{Sobolev Spaces}, the inclusion $H^1(M) \hookrightarrow C^0(M)$ is compact when $\text{dim}(M) =1 $, so $H^1(M) \hookrightarrow C^0(M) \hookrightarrow L^2_\R(M)$ is compact, and the Poincaré inequality still holds.} Using (\ref{B = integral of abs of nabla u^2}), the Poincaré inequality implies
$$\int_M(v - \overline{v})^2 \leq A \cdot B(v,v)$$
for some $A>0$ independent of $v$. Thus,
\begin{equation*}
\begin{split}
\norm{u}_{H^1}^2 & \leq \int_M (u - \overline{u})^2 + B(u,u)\\
& \leq A \cdot B(u,u) + B(u,u)\\
& = (A + 1) B(u,u).\\
\end{split}
\end{equation*}
Dividing over by $(A+1)$ gives
$$\widetilde{B}(u,u) \geq \frac{1}{A+1}\norm{u}^2_{H^1}$$
for all $u \in K^\perp$, showing that $\widetilde{B}$ is coercive.
\end{proof}

We can now utilize the Lax--Milgram theorem. For each $h \in L_\R^2(M)$, define the linear functional $f_h:K^\perp \to \R$ by 
$$f_h(v) = \int_M h \cdot v = \langle h, v \rangle_{L^2},$$
which is bounded by the Cauchy--Schwarz inequality for $L^2$. By the Lax--Milgram theorem, for each $h \in L_\R^2(M)$ there exists a unique $u_h \in K^\perp$ such that 
$$B(u_h,v) = f_h(v)$$
for all $v \in K^\perp$. Let $D^{-1}: L_\R^2(M) \to K^\perp$ denote the function that maps $h \mapsto u_h$ as above. $D^{-1}$ is the inverse Laplace operator.

To see that $D^{-1}$ is linear, note that 
$$B\left(D^{-1}(h+j), \: v\right) = f_{h+j}(v) = f_h(v) + f_j(v) = B\left(D^{-1}\left(h\right) + D^{-1} (j), \: v\right)$$
so by the uniqueness of $D^{-1}(h+j)$, we must have $D^{-1}(h+j) = D^{-1}(h) + D^{-1}(j)$. Similarly, we get $D^{-1}(c \cdot h) = c \cdot D^{-1}(h)$ for $c \in \R$ from the equality $f_{ch}(v) = cf_h(v)$.

From the Lax--Milgram theorem, 
$$\norm{D^{-1}(h)}_{H^1} \leq \frac{1}{A+1}\norm{f_h} = \frac{1}{A+1} \sup \set{\: \left| f_h(v) \right| : v \in K^\perp, \; \norm{v}_{H^1}=1}.$$
From Cauchy--Schwarz and (\ref{compare L^2 and H^k}),
$$\left| f_h(v) \right| =  \left| \langle h,v \rangle_{L^2} \right| \leq \norm{h}_{L^2} \cdot \norm{v}_{H^1},$$
so
\begin{equation*}
\begin{split}
\norm{D^{-1}(h)}_{H^1} & \leq \frac{1}{A+1} \sup \set{\: \left| f_h(v) \right| : v \in K^\perp, \; \norm{v}_{H^1}=1}\\
& \leq \frac{1}{A+1}\sup \set{\: \norm{h}_{L^2} \cdot \norm{v}_{H^1} : v \in K^\perp, \; \norm{v}_{H^1}=1}\\
& = \frac{1}{A+1}\norm{h}_{L^2}.
\end{split}
\end{equation*}
Therefore, $D^{-1}$ is bounded. We may extend the codomain of $D^{-1}$ and consider it as a bounded linear operator $D^{-1}: L_\R^2(M) \to L_\R^2(M)$ since the inclusion $K^\perp \hookrightarrow L_\R^2(M)$ is continuous.

We now extend $D^{-1}$ to complex-valued functions linearly. That is, we define \\$D^{-1}:L^2_\C(M) \to L^2_\C(M)$ in the following way. For $f = f_0 + if_1 \in L^2_\C(M)$ where $f_0, f_1 \in L^2_\R(M)$ are real-valued, we set
$$D^{-1}(f) := D^{-1}(f_0) + i\cdot D^{-1}(f_1).$$
It is easy to see that $D^{-1}$ is linear on $L^2_\C(M)$. To see that it is bounded, note that on $L^2_\C(M)$,
$$D^{-1} = \left(\text{inc}_{L^2_\C} \circ D^{-1} \circ \text{Re}\right) + 
  i\cdot \left( \text{inc}_{L^2_\C} \circ D^{-1} \circ \text{Im}\right)$$
where $\text{Re}:L^2_\C(M) \to L^2_\R(M)$ and $\text{Im}:L^2_\C(M) \to L^2_\R(M)$ map a square-integrable function to its real and imaginary part respectively, and $\text{inc}_{L^2_{\C}}: L^2_{\R}(M) \to L^2_{\C}(M)$ is the natural inclusion of $L^2_\R$ into $L^2_\C$. Since $\text{Re}$, \text{Im}, $\text{inc}_{L^2_\C}$, and \\$D^{-1}:L^2_\R(M) \to L^2_\R(M)$ are all continuous functions, and the linear combination of continuous functions on a Hilbert space is again a continuous function, it follows that $D^{-1}$ is continuous on $L^2_\C(M)$. 

We summarize the above derivation in the following definition.

\begin{defin}[Inverse Laplace Operator]
$D^{-1}:L^2_\C(M) \to L^2_\C(M)$ is the bounded, linear operator given as the $\C$-linear extension to $L^2_\C(M)$ of the operator $D^{-1}:L^2_\R(M) \to L^2_\R(M)$. For each $h \in L^2_\R(M)$, $D^{-1}(h)$ is the unique function in $K^\perp$ satisfying
$$B\left(D^{-1}(h), v\right) = \int_M h \cdot v$$
for all $v \in K^\perp$, where $K^\perp$ is the orthogonal compliment of (\ref{definition of K}) in $H^1(M)$.
$D^{-1}$ is referred to as the \textit{inverse Laplace operator}, or the \textit{inverse Laplacian}.
\end{defin}

\begin{lemma}
$K \subset \text{Ker}(D^{-1})$ where $\text{Ker}(D^{-1})$ is the kernel of the inverse Laplacian. \label{K subset KerD}
\end{lemma}
\begin{proof}
Let $u \in K$. For all $v \in K^\perp$, we have $B(u,v)= 0$ by definition of $K$. Hence,
\begin{equation*}
\begin{split}
\int_M u \cdot v & = \int_M u \cdot v + B(u,v)\\
& = \langle  u,v \rangle_{H^1}\\
& = 0
\end{split}
\end{equation*}
by definition of $K^\perp$. Thus,
$$B(0,v) = \int_M u \cdot v,$$
implying that $D^{-1}(u) = 0.$
\end{proof}

Just as was done for the inverse Laplace operator, we can extend the Laplacian to complex valued-functions linearly. Formally, we define the complex vector space $H_\C^2(M) := \set{f_0 + if_1 \; : \; f_0,f_1 \in H^2(M)}$ and set 
$$\Delta (f_0 + if_1) = \Delta (f_0) + i \Delta(f_1)$$
to get $\Delta:H_\C^2(M) \to L^2_\C(M)$.

\begin{prop}
    Let $f \in H^2_\C(M)$ be an eigenfunction for $\Delta$ with eigenvalue $\lambda \in \R$. If $\lambda \neq 0$, then $f$ is an eigenfunction for $D^{-1}$ with corresponding eigenvalue $\frac{1}{\lambda}$. If $\lambda = 0$, then $f$ is an eigenfunction for $D^{-1}$ with eigenvalue $\lambda = 0$.\label{eigfunctions of inverse laplace}
\end{prop}

\begin{proof}
Let $f = f_0 + if_1$ where $f_0$ and $f_1$ are the real and imaginary parts of $f$ respectively. Since $\lambda$ is real-valued, it follows that $f_0$ and $f_1$ are both eigenfunctions of $\Delta$ with eigenvalue $\lambda$ by taking the real and imaginary parts of \\$\Delta(f) = \Delta(f_0) + i\Delta(f_1) = \lambda(f_0 + if_1)$.

First suppose that $\lambda = 0$. Then $f_0,f_1 \in \text{Ker}(\Delta)$. It's clear from proposition \ref{connecting B and Laplace on H^2} that $\text{Ker}(\Delta) \subset K$, so from lemma \ref{K subset KerD} we have $f_0,f_1 \in \text{Ker}(D^{-1})$.

Now suppose $\lambda \neq 0$. Let $w \in K$. Then $B(f_0,w) = 0$, so $\langle f_0,w \rangle_{H^1} = \int_M f_0 w$. Since $f_0$ is an eigenfunction for $\Delta$, we also have 
\begin{equation*}
\begin{split}
\langle f_0,w \rangle_{H^1} & = \int_M f_0 \cdot w + \int_M \Delta f_0 \cdot w\\
& = (1 + \lambda)\int_M f_0 \cdot w,
\end{split}
\end{equation*}
implying that $\int_M f_0 w = (1 + \lambda)\int_M f_0 w \implies \int_M f_0 w = 0$. Therefore, \\$\langle f_0,w \rangle_{H^1} = 0$, so $f_0 \in K^\perp$ as $w \in K$ was arbitrary. Likewise, it follows that $f_1 \in K^\perp$. In particular, we have $\frac{1}{\lambda}f_0, \frac{1}{\lambda}f_1 \in K^\perp$.

For any $v \in K^\perp,$
\begin{equation*}
\begin{split}
 B\left(\frac{1}{\lambda}f_0, \: v\right) & = \int_M \Delta\left(\frac{1}{\lambda}f_0\right) \cdot v\\
 & = \int_M f_0 \cdot v,   
\end{split}
\end{equation*}
so the uniqueness condition defining $D^{-1}$ implies that $D^{-1}(f_0) = \frac{1}{\lambda}f_0$. Similarly, \\$D^{-1}(f_1) = \frac{1}{\lambda}f_1$. Therefore,
\begin{equation*}
\begin{split}
 D^{-1}(f) & = D^{-1}(f_0) + i \cdot D^{-1}(f_1)\\
 & = \frac{1}{\lambda}(f_0 + if_1)\\
 & = \frac{1}{\lambda}f.
\end{split}
\end{equation*}
\end{proof}

\section{Traces Involving the Inverse Laplacian}\label{traces involving inverse laplace}

\subsection{The $S^1$ Case}

Let's start by considering the inverse Laplacian defined on $S^1.$ Via an application of the Stone--Weierstrass theorem, and the fact that continuous functions are dense in $L^2(S^1)$, it follows that the functions \\$\set{\theta \mapsto \frac{1}{\sqrt{2\pi}}e^{i k \theta}}_{k \in \Z} \subset C^\infty(S^1)$ form an orthonormal basis for $L^2(S^1)$. Also, from example \ref{S^1 laplace}, we know that 
$$\Delta_{S^1}\left(\frac{1}{\sqrt{2\pi}}e^{i k \theta}\right) = \dfrac{d^2}{d\theta^2}\left(\frac{1}{\sqrt{2\pi}}e^{i k \theta}\right) = -k^2\frac{1}{\sqrt{2\pi}}e^{i k \theta}.$$
Thus, from proposition \ref{eigfunctions of inverse laplace}, it follows that for $k \neq 0$
$$D^{-1}_{S^1}\left(\frac{1}{\sqrt{2\pi}}e^{i k \theta}\right) = -\frac{1}{k^2} \cdot \frac{1}{\sqrt{2\pi}}e^{i k \theta}$$
and for $k = 0$, 
$$D_{S^1}^{-1} \left(\frac{1}{\sqrt{2\pi}}\right) = 0.$$
Thus, the orthonormal basis $\set{\frac{1}{\sqrt{2\pi}}e^{i k \theta}}_{k \in \Z} \subset L^2(S^1)$ is an eigenbasis with respect to $D_{S^1}^{-1}$ with corresponding eigenvalues $$\set{0} \bigcup \set{-\frac{1}{k^2}}_{k \in \Z \setminus 0}.$$
It follows that the $n$-fold composition $D_{S^1}^{-n} := (D_{S^1}^{-1})^n = D_{S^1}^{-1} \circ \cdots \circ D_{S^{1}}^{-1}$ has eigenvalues
$$\set{0} \bigcup \set{ \frac{(-1)^n}{k^{2n}}}_{k \in \Z \setminus 0}$$
with respect to the same basis. By proposition \ref{compute trace with eigvals}, in order to show that $D^{-n}_{S^1}$ is trace class, it suffices to show that $\sum_{k \in \Z \setminus 0} (-1)^n \cdot \frac{1}{k^{2n}}$ converges absolutely.

\begin{prop}
$D^{-n}_{S^1}$ is trace class for all $n \in \N$ with 
$$\text{Tr}(D_{S^1}^{-n}) = 2(-1)^n \cdot \zeta(2n)$$
where $\zeta(s) = \sum_{j=1}^\infty \dfrac{1}{j^s}$ is the Riemann zeta function.
\end{prop}
\begin{proof}
The series $\sum_{k \in \Z \setminus 0} \frac{1}{k^{2n}}$ converges absolutely to $2 \zeta(2n)$, which is seen by ordering the terms in the series according to the mapping $\N \to \Z \setminus 0$ given by $1 \mapsto 1, 2 \mapsto -1, 3 \mapsto 2, 4 \mapsto -2,...$ etc. Since $2n>1$, $\zeta(2n) < \infty$. Hence, it clearly follows that $\sum_{k \in \Z \setminus 0} (-1)^n \cdot \frac{1}{k^{2n}}$ converges absolutely to $(-1)^n \cdot 2\zeta(2n) < \infty.$
\end{proof}

\subsection{The $S^1 \times S^1$ Case}

Let's now consider a slightly more difficult example- the inverse Laplace on the torus, $S^1 \times S^1$. Similar to before, the functions \\$\set{(\theta_1,\theta_2) \mapsto \frac{1}{2\pi} e^{ik\theta_1} \cdot e^{im\theta_2}}_{(k,m) \in \Z^2}$ form an orthonormal basis for $L^2(S^1 \times S^1)$. 

Recall from example \ref{torus laplace} that $$\Delta_{S^1 \times S^1}\left(\frac{1}{2\pi} e^{ik\theta_1} \cdot e^{im\theta_2}\right) = -(k^2 + m^2) \frac{1}{2\pi} e^{ik\theta_1} \cdot e^{im\theta_2},$$
so for $(k,m) \neq 0$,
$$D^{-1}_{S^1 \times S^1}\left(\frac{1}{2\pi} e^{ik\theta_1} \cdot e^{im\theta_2}\right) = -\dfrac{1}{k^2 + m^2} \cdot \left(\frac{1}{2\pi} e^{ik\theta_1} \cdot e^{im\theta_2}\right),$$
and for $k = m = 0$,
$$D_{S^1 \times S^1}^{-1}\left(\frac{1}{2\pi}\right) = 0.$$
The eigenvalues of $D_{S^1 \times S^1}^{-n}$ corresponding to the basis $\set{\frac{1}{2\pi} e^{ik\theta_1} \cdot e^{im\theta_2}}_{(k,m) \in \Z^2}$ are then 
$$\set{0} \bigcup \set{\dfrac{(-1)^n}{(k^2 + m^2)^{n}}}_{(k,m) \in \Z^2\setminus (0,0)}.$$

\begin{prop}
For $n \geq 2$, the sum 
$$\sum_{(k,m) \in \Z^2 \setminus{(0,0)}} \dfrac{1}{(k^2 + m^2)^n}$$ converges absolutely with 
$$\sum_{(k,m) \in \Z^2 \setminus{(0,0)}} \dfrac{1}{(k^2 + m^2)^n} = 4 \zeta(n) \beta(n)$$
where $\zeta(n)$ is the Riemann zeta function and $\beta(n) = \sum_{m=0}^\infty (-1)^m (2m+1)^{-n}.$\label{lattice sum example}
\end{prop}

\begin{proof}
This is a classic case of a ``lattice sum'' that appears in \cite{Lattice Sums} on pages 31-33, evaluated using theta functions and the Mellin transform. Let us first define the Mellin transform of a function $f(t)$ as
$$M_s(f(t)) = \dfrac{1}{\Gamma(s)}\int_{0}^\infty t^{s-1}f(t)dt$$ for whatever $s \in \C$ it is defined, where $\Gamma(s)$ is the gamma function. \footnote{It is common to define $M_s$ without the $\Gamma(s)$ factor. However, it is convenient to include it for this computation.} Also, let's define the function $\theta_3$ as 
$$\theta_3 = \sum_{n=-\infty}^\infty q^{n^2}$$
for all $q$ such that the sum is well-defined. Unless stated otherwise, let \\$(k,m) \in \Z^2 \setminus (0,0)$. 
Note that we have the equality 
$$\sum_{(k,m)}(k^2 + m^2)^{-n} = M_n\left( \sum_{(k,m)}e^{-(k^2 + m^2)t} \right)$$
since
\begin{equation*}
\begin{split}
M_n\left( \sum_{(k,m)}e^{-(k^2 + m^2)t} \right)  & = \dfrac{1}{\Gamma(n)}\int_{0}^\infty t^{n-1}\sum_{(k,m)}e^{-(k^2 + m^2)t}dt\\
& = \dfrac{1}{\Gamma(n)}\sum_{(k,m)}\int_{0}^\infty t^{n-1} e^{-(k^2 + m^2)t}dt,
\end{split}
\end{equation*}
and by making the substitution $u = (k^2 + m^2)t$, we get
\begin{equation*}
\begin{split}
\dfrac{1}{\Gamma(n)}\sum_{(k,m)}\int_{0}^\infty t^{n-1} e^{-(k^2 + m^2)t}dt & = \dfrac{1}{\Gamma(n)}\sum_{(k,m)} \int_0^\infty \left(\dfrac{u}{(k^2 + m^2)}\right)^{n-1} e^{-u}\dfrac{du}{(k^2 + m^2)}\\
& = \dfrac{1}{\Gamma(n)}  \sum_{(k,m)} \dfrac{1}{(k^2+m^2)^n} \int_0^\infty u^{n-1}e^{-u}du\\
& = \dfrac{1}{\Gamma(n)} \sum_{(k,m)} \dfrac{1}{(k^2+m^2)^n} \Gamma(n)\\
& = \sum_{(k,m)} \dfrac{1}{(k^2+m^2)^n}.
\end{split}
\end{equation*}
Also, note that
$$\theta_3^2 = \sum_{k=-\infty}^\infty q^{k^2} \cdot \sum_{m=-\infty}^\infty q^{m^2} = \sum_{k=-\infty}^\infty \sum_{m=-\infty}^\infty q^{k^2}q^{m^2} = \sum_{(k,m) \in \Z^2}q^{k^2 + m^2},$$ 
so by taking $q = e^{-t}$, it follows that 
$$\theta_3^2 = \sum_{(k,m) \in \Z^2}e^{-(k^2 + m^2)t},$$
implying that
$$\theta^2_3 - 1 = \sum_{(k,m)}e^{-(k^2 + m^2)t}$$
(where, once again, $(k,m) \in \Z^2 \setminus (0,0)$, which is why subtracting by 1 is necessary). Therefore, 
$$\sum_{(k,m)}(k^2 + m^2)^{-n} = M_n(\theta_3^2 - 1).$$

As mentioned in \cite{Lattice Sums}, one can compute $\theta_3^2 -1 = 4 \sum_{k=1}^\infty\sum_{m=0}^\infty q^k(-1)^m q^{2km}$, so 
\begin{equation*}
\begin{split}
    \sum_{(k,m)}(k^2 + m^2)^{-n} & = M_n(\theta^2_3 - 1)\\
    & = M_n\left(4 \sum_{k=1}^\infty \sum_{m=0}^\infty e^{-tk}(-1)^m (e^{-2kmt})\right)\\
    & = \dfrac{4}{\Gamma(n)}\sum_{k=1}^\infty \sum_{m=0}^\infty \int_0^\infty (-1)^m t^{n-1}e^{-k(2m+1)t}dt.
\end{split}
\end{equation*}
Via the substitution $u = k(2m+1)t$ and integrating, we obtain
\begin{equation*}
\begin{split}
\sum_{(k,m)}(k^2 + m^2)^{-n} & = 4 \sum_{k=1}^\infty \sum_{m=0}^\infty \dfrac{1}{k^n} \cdot (-1)^m (2m+1)^{-n}\\
& = 4\sum_{k=1}^\infty \dfrac{1}{k^n} \cdot \sum_{m=0}^\infty (-1)^m(2m+1)^{-n}\\
& = 4\zeta(n)\beta(n)\\
& < \infty
\end{split}
\end{equation*}
since $n \geq 2$.

\end{proof}

\begin{corollary}
$D_{S^1 \times S^1}^{-n}$ is trace class for $n \geq 2$ with 
$$\text{Tr}(D_{S^1 \times S^1}^{-n}) = 4(-1)^n \cdot \zeta(n) \beta(n).$$
\end{corollary}
\begin{proof}
Use proposition \ref{compute trace with eigvals} with proposition \ref{lattice sum example}.\\
\end{proof}

\subsection{The Inverse Laplacian with $d^\ast$ on the Torus}

Building on our last example, we can couple $D^{-1}_{S^1 \times S^1}$ with the operator $d^\ast$, which is the formal adjoint to the exterior derivative viewed as a map $C^\infty(S^1 \times S^1) \to C^\infty(S^1 \times S^1)$, to get the operator $P_{S^1 \times S^1}^n:= (d^\ast \circ D^{-1}_{S^1 \times S^1})^n$ on $L^2(S^1 \times S^1)$ for each $n \in \N$ (see \cite{GMW} for a more detailed description of this operator). 

More precisely, for the flat square torus,
we have
$$d^\ast = \dfrac{\del}{\del \theta_1} + \dfrac{\del}{\del \theta_2}$$ 
on $\Lambda$, where $\Lambda$ is the finite span of $\set{\frac{1}{2\pi} e^{ik\theta_1} \cdot e^{im\theta_2}}_{(k,m) \in \Z^2}$ over $\C$, viewed as a subspace of $L^2(S^1 \times S^1)$. Via an application of the Pythagorean identity, it follows that the operator $d^\ast \circ D^{-1}_{S^1 \times S^1}:\Lambda \to \Lambda$ is bounded. Hence, it admits a unique extension to a bounded linear map $d^\ast \circ D^{-1}_{S^1 \times S^1}: L^2(S^1 \times S^1) \to L^2(S^1 \times S^1)$ which we denote by $P_{S^1 \times S^1}$.

The eigenvalues of $P^n_{S^1 \times S^1}$ corresponding to the basis $\set{\frac{1}{2\pi} e^{ik\theta_1} \cdot e^{im\theta_2}}_{(k,m) \in \Z^2}$ are

$$\set{0} \bigcup \set{(-i)^n \cdot \dfrac{(k+m)^n}{(k^2 + m^2)^n}}_{(k,m) \in \Z^2 \setminus (0,0)}$$
Therefore, our focus will be on summing these eigenvalues in order to determine if $P_{S^1 \times S^1}^n$ has a trace.

\begin{prop}
For any $n \in 2 \N$ with $n > 2$, the operator $P^n_{S^1 \times S^1}$ is trace class with
$$\text{Tr}(P_{S^1 \times S^1}^n) = i^n \sum_{(k,m)} \dfrac{(k+m)^n}{(k^2+m^2)^n}$$
where $(k,m) \in \Z^2 \setminus (0,0)$.\label{torus P trace even}
\end{prop}
\begin{proof}
Unless stated otherwise, let $(k,m) \in \Z^2 \setminus (0,0).$ As per usual, our goal is to show that $\sum_{(k,m)} \dfrac{(k+m)^n}{(k^2 + m^2)^n}$ converges absolutely, and our desired result will follow from proposition \ref{compute trace with eigvals}. First, note that by the Binomial theorem, 

$$\dfrac{(k+m)^n}{(k^2+m^2)^n} = \sum_{j=0}^n \binom{n}{j}  \dfrac{k^{n-j}m^j}{(k^2+m^2)^n}.$$
Hence, if it is shown that the series $\sum_{(k,m)} \frac{k^{n-j}m^j}{(k^2+m^2)^n}$ converges absolutely for each $0 \leq j \leq n$, then the proof will be complete.\\

Suppose that $j$ is even. Then $n-j$ is also even, implying that the terms in the sum above are all nonnegative. Also, $(j/2) \in \Z$ with $0 \leq \frac{j}{2} \leq \frac{n}{2}$, so $\binom{n/2}{j/2}(k^2)^{\frac{n}{2} - \frac{j}{2}}(m^2)^{\frac{j}{2}}$ is a term in the sum $\sum_{i=0}^{n/2} \binom{n/2}{i} (k^2)^{\frac{n}{2} - i} (m^2)^{i}.$ In particular, $$\binom{n/2}{j/2}(k^2)^{\frac{n}{2} - \frac{j}{2}}(m^2)^{\frac{j}{2}} \leq \sum_{i=0}^{n/2} \binom{n/2}{i} (k^2)^{\frac{n}{2} - i} (m^2)^{i}.$$ We can utilize the Binomial theorem once more and expand $(k^2 + m^2)^{n/2} = \sum_{i=0}^{n/2} \binom{n/2}{i} (k^2)^{\frac{n}{2} - i} (m^2)^{i}$. Thus,

$$k^{n-j}m^j \leq \binom{n/2}{j/2}(k^2)^{\frac{n}{2} - \frac{j}{2}}(m^2)^{\frac{j}{2}} \leq \sum_{i=0}^{n/2} \binom{n/2}{i} (k^2)^{\frac{n}{2} - i} (m^2)^{i} = (k^2 + m^2)^{n/2},$$
which implies that
$$\sum_{(k,m)} \dfrac{k^{n-j}m^j}{(k^2 + m^2)^n} \leq \sum_{(k,m)} \dfrac{(k^2 + m^2)^{n/2}}{(k^2 + m^2)^n} = 
\sum_{(k,m)}\dfrac{1}{(k^2 + m^2)^{n/2}} < \infty.$$
The convergence of $\sum_{(k,m)}\dfrac{1}{(k^2 + m^2)^{n/2}}$ follows from proposition \ref{lattice sum example} since \\$n/2 \geq 2.$

Suppose now that $j$ is odd. Then $j-1$ and $j+1$ are even such that \\$0 \leq \frac{j-1}{2}, \: \frac{j+1}{2} \leq \frac{n}{2}$, so 
$$\binom{n/2}{(j-1)/2}(k^2)^{\frac{n}{2} - \frac{j-1}{2}}(m^2)^{\frac{j-1}{2}}$$ and $$\binom{n/2}{(j+1)/2}(k^2)^{\frac{n}{2} - \frac{j+1}{2}}(m^2)^{\frac{j+1}{2}}$$ are both terms in the sum $\sum_{i=0}^{n/2} \binom{n/2}{i} (k^2)^{\frac{n}{2} - i} (m^2)^i$. Thus, 
\begin{equation*}
\begin{split}
k^{n-(j-1)}m^{j-1} + k^{n-(j+1)}m^{j+1} & \leq \binom{n/2}{(j-1)/2}(k^2)^{\frac{n}{2} - \frac{j-1}{2}}(m^2)^{\frac{j-1}{2}} \\
& + \binom{n/2}{(j+1)/2}(k^2)^{\frac{n}{2} - \frac{j+1}{2}}(m^2)^{\frac{j+1}{2}}\\ 
& \leq \sum_{i=0}^{n/2} \binom{n/2}{i} (k^2)^{\frac{n}{2} - i} (m^2)^i\\ 
& = (k^2 + m^2)^{n/2}.
\end{split}
\end{equation*}
If we can show that $|k^{n-j}m^j| \leq k^{n-(j-1)}m^{j-1} + k^{n-(j+1)}m^{j+1}$, then we can make the same comparison with the series $\sum_{(k,m)}(k^2+m^2)^{-n/2}$ applied in the previous case, and the proof will be complete. Note that $|km| \leq k^2 + m^2$. Therefore, 
\begin{equation*}
\begin{split}
|k^{n-j}m^j| & = k^{n-j-1}m^{j-1}|km| \\
& \leq k^{n-j-1}m^{j-1}(k^2 + m^2)\\ 
& = k^{n-(j-1)}m^{j-1} + k^{n-(j+1)}m^{j+1}\\
& \leq (k^2 + m^2)^{n/2},
\end{split}
\end{equation*}
so
$$\sum_{(k,m)}\left| \dfrac{k^{n-j}m^j}{(k^2 + m^2)^n} \right| \leq \sum_{(k,m)}  \dfrac{1}{(k^2 + m^2)^{n/2}} < \infty.$$

Since each case has been considered, it follows that $\sum_{(k,m)}\dfrac{(k+m)^n}{(k^2 + m^2)^n}$ converges absolutely $\implies P_{S^1 \times S^1}^n$ has a well-defined trace.\\  
\end{proof}

To the extent possible, the sum appearing in proposition \ref{torus P trace even} is evaluated in \cite{GMW}. For the boundary case $(n=2)$ it turns out that the sum does not converge. That is, the trace of the operator $P_{S^1 \times S^1}^2$ does not exist.  

\begin{prop}
    The operator $P^2_{S^1 \times S^1}$ is not trace class.
\end{prop}

\begin{proof}
   Fix an ordering of the terms for the series $\sum_{(k,m) \in \Z^2 \setminus (0,0)} \dfrac{(k+m)^2}{(k^2+m^2)^2}$. Let 
   $$\sum_{n=1}^M a_{k_n,m_n}$$
   denote the $M^{th}$ partial sum of our series with respect to this ordering. That is, 
   $$a_{k_n,m_n} = \dfrac{(k_n + m_n)^2}{(k_n^2 + m_n^2)^2}.$$
   Now let $N > 0$. Without loss of generality, we may take $N \in \N$ as our goal is to show $\sum_{n=1}^M a_{k_n,m_n} \geq N$ for some $M$. Choose $M$ large enough such that, for all $(k,m) \in \Z^2 \setminus (0,0)$ with $2 \leq k,m \leq 2^{N+1}-1$, there exists an $n \leq M$ such that $k_n = k$ and $m_n = m$. That is, we are choosing $M$ sufficiently large such that all of the terms in our sum satisfying $2 \leq k,m \leq 2^{N+1}-1$ have been summed via the $M^{th}$ partial sum. Since the intervals $[2^j,2^{j+1}-1]$ for $j \in \set{1,...,N}$ are mutually disjoint (in particular, the sets formed by taking only the integer values from these intervals are mutually disjoint) it follows that 
   $$\sum_{n=1}^M a_{k_n,m_n} \geq \sum_{j=1}^N \: \sum_{k,m = 2^j}^{2^{j+1}-1} \dfrac{(k+m)^2}{(k^2 + m^2)^2}$$
    since each of the terms in the sum on the right hand side satisfy the condition in which $M$ was chosen, and we are not double counting any terms due to the disjoint condition just mentioned. Note that 
    \begin{equation*}
    \begin{split}
     \sum_{j=1}^N \: \sum_{k,m = 2^j}^{2^{j+1}-1} \dfrac{(k+m)^2}{(k^2 + m^2)^2} & \geq \sum_{j=1}^N \sum_{k,m = 2^j}^{2^{j+1}-1}\dfrac{(2^j + 2^j)^2}{(2^{2j} + 2^{2j})^2}\\
     & = \sum_{j=1}^N \sum_{k,m = 2^j}^{2^{j+1}-1}\dfrac{2^{2j + 2}}{2^{4j + 2}} \\
     & = \sum_{j=1}^N \sum_{k,m = 2^j}^{2^{j+1}-1} \dfrac{1}{2^{2j}}.
    \end{split}
    \end{equation*}
    Since $\sum_{k,m = 2^j}^{2^{j+1}-1}\dfrac{1}{2^{2j}}$ does not depend on $k$ or $m$, it is equal to some multiple of the summand $\dfrac{1}{2^{2j}}$. The lower and upper bounds on the sum were specifically chosen such that this multiple would be $2^{2j}$. Hence,
    \begin{equation*}
    \begin{split}
     \sum_{n=1}^M a_{k_n,m_n} & \geq \sum_{j=1}^N  \sum_{k,m = 2^j}^{2^{j+1}-1}\dfrac{1}{2^{2j}}\\
      & = \sum_{j=1}^N 2^{2j} \cdot \dfrac{1}{2^{2j}}\\
      & = \sum_{j=1}^N 1\\
      & = N.
    \end{split}
    \end{equation*}
    Since $N > 0$ was arbitrary, $\sum_{n=1}^\infty a_{k_n, m_n} = \infty$, so 
    $$\sum_{(k,m) \in \Z^2 \setminus(0,0)} \dfrac{(k+m)^2}{(k^2+m^2)^2} = \infty.$$
\end{proof}

\textbf{Acknowledgments} This article derives from research partially funded by the Undergraduate Scholars Program at Montana State University. Also, I would like to thank Dr. Grady for his advisory on the project for which the aforementioned research took place.

\end{document}